\documentclass[10pt,a4paper,oneside,reqno]{amsart}
\input{structureAMS.tex} 

\newcommand{\1}{\mathds{1}}
\newcommand{\R}{\mathbb{R}}  
\ifdefined\C\renewcommand{\C}{\mathbb{C}}\else\newcommand{\C}{\mathbb{C}}\fi 
\newcommand{\N}{\mathbb{N}}  
\newcommand{\deq}{%
  \mathrel{\vbox{\offinterlineskip\ialign{%
    \hfil##\hfil\cr
    $\scriptscriptstyle d$\cr
    \noalign{\kern.1ex}
    $=$\cr
}}}}

\newcommand{\landauO}[1]{\mathcal{O}\left(#1\right)} 
\newcommand{\stO}[1]{\mathcal{O}_\prec\left(#1\right)} 

\newcommand*{\defeq}{\mathrel{\vcenter{\baselineskip0.5ex \lineskiplimit0pt\hbox{\scriptsize.}\hbox{\scriptsize.}}}=}
                     
\providecommand{\norm}[1]{\left\lVert#1\right\lVert} 
\providecommand{\abs}[1]{\left\lvert#1\right\rvert} 
\providecommand{\braket}[1]{\left\langle#1\right\rangle} 
\providecommand{\Cov}[1]{\mathbf{Cov}\left(#1\right)} 

\DeclareMathOperator{\arctanh}{arctanh}

\DeclareMathOperator{\E}{\mathbf{E}}
\DeclareMathOperator{\Var}{\mathbf{Var}}
\renewcommand{\P}{\mathbf{P}}
\DeclareMathOperator{\Tr}{Tr}

\newcommand\restr[2]{{
  \left.\kern-\nulldelimiterspace 
  #1 
  \vphantom{\big|} 
  \right|_{#2} 
  }}
\makeatletter
\providecommand*{\diff}%
        {\@ifnextchar^{\DIfF}{\DIfF^{}}}
\def\DIfF^#1{%
        \mathop{\mathrm{\mathstrut d}}%
                \nolimits^{#1}\gobblespace
}
\def\gobblespace{%
        \futurelet\diffarg\opspace}
\def\opspace{\let\DiffSpace\! \ifx\diffarg(\let\DiffSpace\relax\else\ifx\diffarg\let\DiffSpace\relax\else\ifx\diffarg\{\let\DiffSpace\relax\fi\fi\fi\DiffSpace}   

\title{Fluctuations of Rectangular Young Diagrams of Interlacing Wigner Eigenvalues}

\author{L\'aszl\'o Erd\H{o}s$^\ast$ \quad Dominik Schr\"{o}der$^{\ast\dagger}$}
\address{IST Austria, Am Campus 1, 3400 Klosterneuburg, Austria}
\email{lerdos@ist.ac.at}
\email{dschroed@ist.ac.at}
\thanks{$^\ast$ Partially supported by ERC Advanced Grant No. 338804}
\thanks{$^\dagger$ Partially supported by the IST Austria Excellence Scholarship}

\date{\today}

\begin{document}

\maketitle
\begin{abstract}
 We prove a new CLT for the  \emph{difference} of linear eigenvalue statistics of a Wigner random matrix $H$ and its minor $\widehat H$
and find that the fluctuation is much smaller than the  fluctuations of the individual linear statistics, 
as a consequence of the strong correlation between the eigenvalues of $H$ and  $\widehat H$. In particular our theorem identifies 
the fluctuation  of Kerov's  rectangular Young diagrams, defined by
 the interlacing eigenvalues of $H$ and $\widehat H$, around their asymptotic shape, the Vershik-Kerov-Logan-Shepp curve. Young diagrams equipped with the Plancherel measure follow the same limiting shape.
 For this, algebraically motivated, ensemble a CLT has been obtained in \cite{Ivanov2002} which is structurally similar to our result but the variance is different,  indicating that the analogy between the two models has its limitations.
 Moreover, our theorem shows that Borodin's result \cite{Bor1} on the convergence of the spectral distribution of Wigner matrices to a Gaussian free field also holds in derivative sense. 
\end{abstract}

\section{Introduction}
There is a rich history of probabilistic models of essentially algebraic nature with surprising connections to random matrix theory. 
Examples include the longest increasing  subsequence in random permutations \cite{1998math.....10105B},
queuing processes \cite{Baryshnikov99guesand}, random tilings of a hexagon \cite{johansson2006eigenvalues}, poly-nuclear growth processes \cite{PhysRevLett.84.4882} and 1+1 dimensional exclusion processes
(see e.g. \cite{2012arXiv1212.3351B}  for a good overview of the topic).
 Recent years have seen a spectacular progress towards the KPZ universality that
is detected  in the extreme regimes. The intuition for the KPZ universality often comes from relating 
these model to the extreme eigenvalues of  random matrices.  Many of these models are related to a classical algebraic problem, the statistics of  Young tableaux from 
representation theory. 
 In this paper we focus on the bulk regime and we investigate the analogy between large Young diagrams equipped with the classical Plancherel measure
  and Kerov's \emph{rectangular Young diagrams}, originating
 from eigenvalue statistics of  minors of large random Wigner random matrices.  Their limiting shape curves   coincide.
 Here we identify the fluctuation of the rectangular Young diagrams
and establish the precise conditions when it is Gaussian and we compute
its correlation. We find  that the limiting behavior of the two diagram ensembles are not the same, even though in the extreme regime their statistics coincide.

Given an integer $N$ and a partition $N=\lambda_1+\lambda_2+\dots$ of $N$ into integers $\lambda_1\geq \lambda_2\geq\dots\geq 0$, the planar figure obtained as the union of consecutive rows consisting of $\lambda_1,\lambda_2,\dots$ unit square cells, is called the \emph{Young diagram} corresponding to $\lambda$ of size $\abs{\lambda}=N$.  Young diagrams of size $N$ are commonly considered as a probability space 
equipped with the \emph{Plancherel measure} $P_N(\lambda)\defeq d_\lambda^2/N!$, where $d_\lambda$ is the number of Young tableaux with given shape $\lambda$ (see, e.g. \cite{fulton1997young}).

The first major connection between Young diagrams and random matrix theory was established by Baik, Deift and Johansson who showed  in \cite{1998math.....10105B} that the distribution of $\lambda_1/\sqrt{N}$ with respect to $P_N(\lambda)$ asymptotically agrees with the distribution of the largest eigenvalue of an $N\times N$
 GUE matrix, hence it follows the Tracy Widom law \cite{TW}. Similar result  \cite{1999math......1118B} holds for
 $\lambda_2/\sqrt{N}$ and the second largest eigenvalue, and Okounkov \cite{Okounkov01012000}
 established  that the joint distribution of $\lambda_1/\sqrt{N},\ldots,\lambda_k/\sqrt{N}$ asymptotically follows that of the $k$ 
 largest eigenvalues of the GUE. Alternative proofs  are given in \cite{1999math......5032B,1999math......6120J}. In fact, in \cite{1999math......5032B} also sine kernel universality in the bulk regime (that is, correlation functions of rows $\lambda_k$ with $k\sim\sqrt N$) has been proven. 

To study the bulk behavior of Young diagrams, it is convenient to draw them in the Russian convention which is rotated by $45^\circ$ from the horizontal convention (see Figure \ref{fig:Youngtab}). In this way we can view the upper boundary the diagram as a continuous function $E\mapsto \lambda(E)$ such that $\lambda(E)\geq \abs{E}$ and $\lambda'(E)=\pm 1$, whenever it is defined. We can continuously extend this function by $\lambda(E)=\abs{E}$ outside the extent of the diagram. 
The limiting shape and the fluctuation of this curve under the Plancherel measure, after proper rescaling, has been determined:
\begin{align}\frac{1}{\sqrt{N}}\lambda(\sqrt{N}E) \approx \Omega(E)+\frac{2}{\sqrt{N}}\Delta(E),
\qquad N\to \infty, \label{eq:KerovCLT}\end{align}
 where  \[\Omega(E)\defeq \begin{cases}
\abs{E}&\text{if }\abs{E}\geq 2\\
\frac{2}{\pi}\left[E\arcsin\frac{E}{2}+\sqrt{4-E^2}\right]&\text{else}
\end{cases} \] is the \emph{Vershik-Kerov-Logan-Shepp curve}. The fluctuation term  $\Delta(E)$ is a generalized Gaussian process on the interval $[-2,2]$  
that can be defined by the trigonometric series \[\Delta(2\cos\theta)=\frac{1}{\pi}\sum_{k\geq 2}\frac{\xi_k\sin k\theta}{\sqrt k}\] of independent standard Gaussian random variables $\xi_k$.  
The limit shape has been independently identified in \cite{LOGAN1977206} and \cite{Kerov}, the fluctuation was proved  in \cite{Ivanov2002}
following Kerov's unpublished notes.

\begin{figure}
\centering
\includegraphics{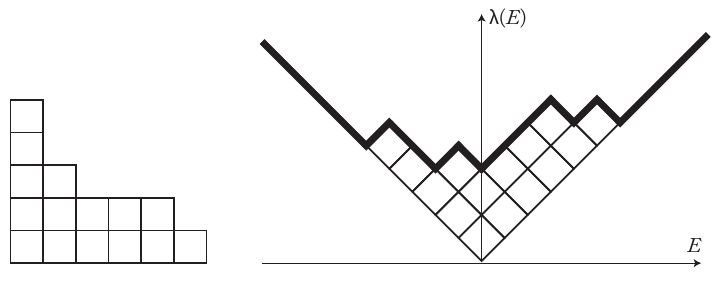}
\caption{Young diagram in French and Russian convention corresponding to the partition $15=6+5+2+1+1$, together with the curve $\lambda(E)$}\label{fig:Youngtab}
\end{figure}

A direct connection between random matrices and Young diagrams in the bulk regime was found by Kerov in \cite{Kerov93theasymptotics}. He showed that for a Wigner random matrix $H\in\C^{N\times N}$ and an independent random $N-1$ dimensional hyperplane $h$ with uniformly distributed normal vector, the eigenvalues $\lambda_1,\dots,\lambda_N$, and $\widehat\lambda_1,\dots,\widehat\lambda_{N-1}$ of $H$ and $P_h H P_h$, where $P_h$ is the projection onto $h$, can be used to construct a curve very similar to Young diagrams. He defined a \emph{rectangular Young diagram} (for a more general context, see \cite{kirillov1998kirillov}) as the function 
\[w_N(E)\defeq\sum_{k=1}^N\abs{\lambda_k-E}-\sum_{k=1}^{N-1}\abs{\widehat\lambda_k-E}, \qquad  E\in \R.  \] It is easy to see that $w_N$ is the unique piecewise-linear continuous function with local minima in $\lambda_1\leq\ldots\leq\lambda_N$ and local maxima in $\widehat{\lambda}_{1}\leq\ldots\leq\widehat{\lambda}_{N-1}$ such that the slope is $\pm 1$ whenever it exists and $w_N(E)=\abs{E-\sum\lambda_k+\sum\widehat\lambda_k}$ for large enough $\abs{E}$. It was then shown that \[\lim_{N\to\infty}\E w_N(E)=\Omega(E),\] uniformly in $E$.

 Bufetov  in \cite{Bufetov2013}  has recently  improved this result in two directions.
 First, he showed  that the randomness in the choice of the projection is not needed; it
is sufficient to consider the eigenvalues of $H$ and its minor $\widehat H=(h_{ij})_{i,j\geq 2}$  (where the choice of removed row/column
 is, of course, arbitrary). Second, he improved the convergence in expectation  to convergence
 in probability;  
 \begin{align}\lim_{N\to\infty}\sup_E\abs{w_N(E) -\Omega(E)}=0.
 \label{eq:bufetov}\end{align} 

We note that $\sum_k \lambda_k = \sum_k \widehat\lambda_{k-1} + h_{11}$, so $w_N(E) = \abs{E-h_{11}}$ for large $E$ and thus it does not exactly match $\Omega(E)$ even outside of the limiting spectrum $[-2,2]$. To remedy this, we will also consider the {\it shifted diagram}
 \[\widetilde w_N(E)\defeq w_N(E+h_{11})\] which agrees identically with $\Omega(E)$ outside the spectrum.  This modification is irrelevant for the limit shape but it becomes relevant when we consider fluctuations.
Figure \ref{fig:exampleDiagrams} shows realizations of $\widetilde w_N$ for different values of $N$ together with the limiting curve $\Omega$. 

In the present work we upgrade the law of large numbers type results \eqref{eq:bufetov} to a central limit theorem (CLT) as in \eqref{eq:KerovCLT}, and thus demonstrate that a certain
 analogy between random matrices and representation theory extends beyond the macroscopic behavior. Specifically, we prove that  \begin{align}\label{YoungCLT}w_N(E)&\approx\Omega(E) + \frac{1}{\sqrt N}\left[ \widehat\Delta(E)+ \xi_{11}\frac{E\sqrt{(4-E^2)_+}+4\arcsin E/2}{2\pi} \right],\\\label{youngCLT2} \widetilde w_N(E)&\approx\Omega(E) + \frac{1}{\sqrt N}\left[ \widehat\Delta(E)+ \xi_{11}\frac{E\sqrt{(4-E^2)_+}}{2\pi} \right]\end{align} where  $\widehat\Delta(E)$ are collections of centered Gaussian random variables whose covariance structure we explicitly compute
  and  $\xi_{11}=\sqrt{N} h_{11}$ is independent of them. Therefore the fluctuations of $w_N$ and $\widetilde w_N$ are Gaussian if and only if $h_{11}$ is Gaussian. We also conclude from our explicit formulas for the variances that although \eqref{YoungCLT} resembles \eqref{eq:KerovCLT}, 
  the distribution of the Gaussian part of the fluctuation, $\widehat\Delta(E)$ and $\Delta(E)$ do not agree. For example -- in contrast to $\Delta(E)$ -- the fluctuation term $\widehat\Delta(E)$ has a finite variance. 

Motivated by the preprint of the current paper, Sasha Sodin \cite{sodin}
considered another rectangular Young diagram $w_N^\ast$ obtained from
the interlacing roots and extrema of the characteristic polynomial of $H$.  
He found that \begin{align}\label{SodinCLT}w_N^\ast(E)\approx\Omega(E)+\frac{1}{N}\widetilde\Delta(E),\end{align} albeit in a weaker sense than \eqref{YoungCLT}, where $\widetilde\Delta(E)$ is a generalized Gaussian process closely related to $\Delta(E)$ in \eqref{eq:KerovCLT}.
In particular the fluctuations of $w_N^\ast$ are always Gaussian; the distribution of any specific matrix entry does not play
 a distinguished role. The difference between $w_N$ and $w_N^\ast$
can be understood via the Markov correspondence (see, e.g.~\cite{kerov1993transition}). Sodin pointed out 
 that the rectangular Young diagram $w_N$ created by a random matrix $H$ and its minor $\widehat H$ is related to the entrywise spectral measure
  $\rho_N$, defined as $\int f\diff \rho_N \defeq f(H)_{11},$ 
 while the empirical spectral density $\mu_N=\frac{1}{N}\sum_k \delta_{\lambda_k}$ corresponds to the rectangular Young diagram $w_N^\ast$.
Thus the behavior of $w_N^\ast$ is directly  related to  $\frac{1}{N}\Tr f(H)$ and not to $f(H)_{11}$ which also explains the difference in the size of the fluctuations between \eqref{YoungCLT} and \eqref{SodinCLT}. For more details on the relation of $w_N$ and $w_N^\ast$ we refer to \cite{sodin}.

\begin{figure} \begin{center}
\includegraphics{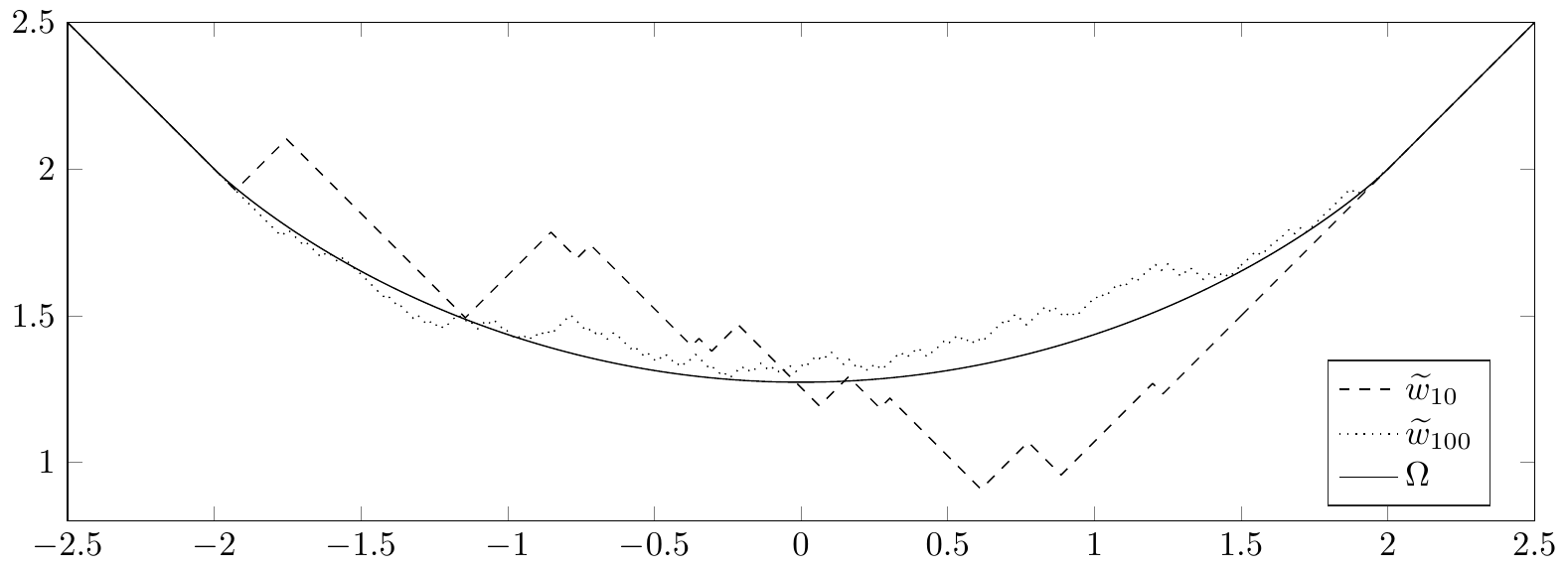}
\end{center} \caption{Sample rectangular Young diagrams $\widetilde w_N$ with limiting shape $\Omega$}\label{fig:exampleDiagrams}
\end{figure}

We prove our results \eqref{YoungCLT}--\eqref{youngCLT2} as corollaries to a new central limit theorem for the \emph{difference} in linear eigenvalue statistics of a Wigner random matrix and its minor. For many classes of random matrices $H=H^{(N)}\in\C^{N\times N}$ the empirical spectral density, i.e., the normalized counting measure of eigenvalues, $\frac{1}{N}\sum_{k=1}^N\delta_{\lambda_k}$ converges weakly to a deterministic measure $\rho$ as $N\to \infty$, which may be
viewed as  type of  law of large numbers. 
 Phrased in terms of an appropriate test function $f$, 
  \[\lim_{N\to\infty} \frac{1}{N}\sum_{k=1}^N f(\lambda_k(H^{(N)}))=\int f(x) \rho(\diff x),
\] 
naturally raises the question whether the fluctuations in this convergence also follow an analogue of the central limit theorem. The object $\sum_{k=1}^N f(\lambda_k(H^{(N)})) = \Tr f(H^{(N)})$, called the \emph{linear eigenvalue statistics} of $H^{(N)}$, has been studied  for many types of random matrices \cite{lytova2009,Anderson2006,Shcherbina2013,2011arXiv1101.3249S,2014arXiv1412.2445J,2013JSP...150...88B,2012arXiv1210.5666S}  
and large classes of test functions $f$. Contrary to the classical CLT, the fluctuations of the linear eigenvalue statistics do not grow with $N$, at least if $f$ is sufficiently regular. The  fluctuations
are typically  Gaussian, but there are also some pathological examples where this is not the case, e.g. for certain invariant ensembles with density supported on several intervals \cite{:/content/aip/journal/jmp/47/10/10.1063/1.2356796}.
For polynomial test functions $f$ the Gaussian fluctuation can be proved by  the elementary moment method, see e.g. 
\cite[Theorem 2.1.13]{AGZ}, but a simple approximation argument does not suffice to extend the result to less regular $f$.
CLT  still holds, for example
  it has been shown in \cite{lytova2009} that for GOE random matrices and test functions $f$ with bounded derivative
 $ \Tr f(H^{(N)}) -\E \Tr f(H^{(N)})$
converges in distribution to a centered Gaussian random variable of variance \begin{align}
\frac{1}{2\pi^2}\int_{-2}^2\int_{-2}^2\left(\frac{f(x)-f(x')}{x-x'}\right)^2\frac{4-xx'}{\sqrt{4-x^2}\sqrt{4-x'^2}}\diff x\diff x'
.\end{align} 
The currently  weakest regularity conditions on $f$ for CLT are found in \cite{2012arXiv1210.5666S}; $f\in H^{1+\epsilon}$ is necessary for general Wigner
matrices and $f\in C^{1/2+\epsilon}$ suffices for GUE. We stress that linear statistics are very sensitive to regularity of the test function; while
polynomial test functions do not require understanding of any local eigenvalue statistics (the global moment method works), the proof
in \cite{2012arXiv1210.5666S} for the  Wigner case heavily relied on techniques developed to prove local semicircle laws \cite{Erdos20121435},
while the GUE case even used the Br\'ezin-Hikami formula and saddle point analysis of the determinantal kernel by Johansson \cite{Joh}.
 
All previous work concerned linear statistics of a single Wigner matrix except 
 two papers by Borodin \cite{Bor1, Bor2} and a few recent works motivated by them. 
  In these papers joint fluctuations of linear statistics of Wigner matrices and its minors were investigated (see also \cite{johnson2014}
where a similar question was discussed for $d$-regular graphs).  Borodin considered
 general families of regularly nested minors and identified the limit  of their joint spectral counting functions
as a Gaussian free field (GFF), but the test function was polynomial
and thus a relatively simple extension of the moment method \cite{AGZ} worked.
The class of test functions was extended to include functions with
 a high Sobolev regularity ($H^{2.5+\epsilon}$ for Gaussian and $H^{5.5+\epsilon}$ for general Wigner matrices) using a Chebyshev basis decomposition
 \cite{2015arXiv150405933L} (see also \cite{kargin2015} where not only nested but
  overlapping matrices were considered).  However, all these results identify the joint fluctuations  on order one scale, whose correlations are typically strictly between 0 and 1  for a collection of  minors whose sizes asymptotically differ by $cN$. Our
work detects the  small fluctuation of order $N^{-1/2}$   resulting  from the very strong correlation between 
 minors of almost the same size. This fine effect is not visible on the scale of the analysis in \cite{Bor1, Bor2, 2015arXiv150405933L}.
 Nevertheless, one may ask whether the fine scale  covariance structure proven in  our main Theorem \ref{thm:mainThm}
is consistent with the covariance formula in \cite{Bor1, Bor2, 2015arXiv150405933L} if one formally applies it
to $H$ and its immediate minor $\widehat H$ disregarding the interchange of limits. Effectively this question 
is equivalent  to asking whether the convergence of the spectral counting functions of the minors to the GFF
also holds  in derivative sense. In Appendix \ref{sec:GFFcomp} we show that the derivative of the GFF
 predicts the correct variance of the fluctuations but fails to identify their distribution, in general. This is essentially due to the fact
 that our fine scale result depends on the precise distribution of $h_{11}$ while the macroscopic formula does not depend on any individual matrix entry. 

Inspired by Kerov's rectangular Young diagrams, in the present work we study the difference of 
two linear statistics $f_N\defeq \Tr f(H)-\Tr f(\widehat H)$ of a Wigner matrix and its minor for a large class of test functions
that includes $f(x) = \abs{x-E}$.
We find that the expectation of $f_N$  converges to \[\Omega_f\defeq \frac{1}{\pi}\int_{-2}^2 \frac{f(x)}{\sqrt{4-x^2}}\diff x\] and its fluctuations around $\Omega_f$ are  of order $N^{-1/2}$. In particular, the fluctuations we identify are much smaller than those  of the
individual  linear statistics, as a result of the strong correlation of the eigenvalues of $H$ and $\widehat H$. Moreover, we prove that the fluctuations are Gaussian if and only if $h_{11}$ follows a normal distribution. It is clear that $h_{11}$ plays a special role, since for example with $f(x)=x$, we have $f_N=\Tr H-\Tr\widehat H=h_{11}$.
Since our test function has a relatively low regularity, our proof requires to understand the spectral statistics on small mesoscopic scales.
In practice, we jointly analyze the Green functions $G(z)=(H-z)^{-1}$ and $\widehat G(z) =(\widehat H-z)^{-1}$ on a spectral scale
$\Im z\ge N^{-2/3}$.

After completing this manuscript, we learned\footnote{Private communication} from Vadim Gorin 
that he and Lingfu Zhang have obtained \cite{gorin2} the exact analogue of our result for the
multilevel extension of the $\beta$-Jacobi ensemble that was introduced in \cite{CPA:CPA21546} as an analogue of the minor process for general $\beta$-ensemble.

\emph{Acknowledgement.} The authors are grateful for discussions with Zhigang Bao and for advice on references from Alexei Borodin. 
 We thank Vadim Gorin for motivating the observation discussed in Appendix \ref{sec:GFFcomp}.

\section{Main Results}
We consider complex Hermitian and real symmetric random matrices and their minors of the form \[ H\defeq\left( \begin{matrix}
h_{11} &\dots& h_{N1}\\
\vdots & \ddots & \vdots\\
h_{1N} &\dots & h_{NN}
\end{matrix} \right),\qquad \widehat{H}\defeq\left( \begin{matrix}
h_{22} &\dots& h_{N2}\\
\vdots & \ddots & \vdots\\
h_{2N} &\dots & h_{NN}\end{matrix} \right)\]
with $(h_{ij})_{i,j=1}^N$ being independent (up to the symmetry constraint $h_{ij}=\overline{h_{ji}}$) random variables satisfying 
\begin{equation}\label{moments}
\E h_{ij}=0,\quad \E \abs{h_{ij}}^2=\frac{s_{ij}}{N}\quad\text{and}\quad \E \abs{h_{ij}}^p \leq \frac{\mu_p}{N^{p/2}}
\end{equation}
for all $i,j,p$ and some absolute constants $\mu_p$. Our main result about the difference of linear eigenvalue statistics of a Wigner random matrix and its minor is as follows.

\begin{theorem}\label{thm:mainThm}
Let the Wigner matrix $H$ satisfy \eqref{moments},  $s_{ij}=1$  for $i\not = j$ and $s_{ii}\le C$ for all $i$, $\E\abs{h_{1j}}^4=\sigma_4/N^2$ for $j=2,\dots,N$ and $\E h_{ij}^2=\sigma_2/N$ for $i<j$. 
 Moreover, let $f\in H^2([-10,10])$
 be some real-valued function. Then the random variables \[f_N\defeq \Tr f(H)-\Tr f(\widehat H)= \sum_{k=1}^N f(\lambda_k)-\sum_{k=1}^{N-1} f(\widehat\lambda_k)\quad\text{and}\quad\widetilde f_N\defeq\sum_{k=1}^N f(\lambda_k-h_{11})-\sum_{k=1}^{N-1} f(\widehat\lambda_k-h_{11})\]  are approximately given by \begin{align}\label{eq:mainEQ}f_N \approx \Omega_f + \frac{1}{\sqrt N} \left[ \Delta_f + \xi_{11}\int_{-2}^2 f'(x)\rho(x)\diff x \right]\quad\text{and}\quad \widetilde f_N \approx \Omega_f + \frac{1}{\sqrt N} \left[ \Delta_f + \xi_{11}\int_{-2}^2 \frac{x f''(x)}{2}\rho(x)\diff x \right],\end{align} where
 where $\rho(x) \defeq \frac{1}{2\pi}\sqrt{4-x^2}$ is the density of the semicircle law,
  \[\Omega_f\defeq \frac{1}{\pi}\int_{-2}^2\frac{f(x)}{\sqrt{4-x^2}}\diff x\] and $\Delta_f$ is a centered Gaussian random variable,
   independent of $\xi_{11}$.  Its variance 
 is given by the explicit formulas
  \begin{align}\nonumber\E (\Delta_f)^2&=V_f\defeq V_{f,1}+\abs{\sigma_2}^2V_{\sigma_2}+(\sigma_4-1)V_{f,2}\\ \label{eq:Vf1}V_{f,1}&=\int_{-2}^2f'(x)^2\rho(x)\diff x-\left(\int_{-2}^2 xf'(x)\rho(x)\diff x\right)^2-\left(\int_{-2}^2f'(x)\rho(x)\diff x\right)^2,\\ \label{eq:Vf2} V_{f,2}&= \left(\int_{-2}^2 xf'(x)\rho(x)\diff x\right)^2,\end{align} where 
  $V_{\sigma_2}$, as defined in eq.~\eqref{eq:Fsigma2}, is a correction term only needed when $\sigma_2\not=0$. For the special case of symmetric Wigner matrices $H$ where $\sigma_2=1$ holds automatically, we have $V_{\sigma_2}=V_{f,1}$.

More precisely, for any fixed $\epsilon>0$, 
\begin{equation}\label{expect}
\E f_N = \Omega_f +\landauO{N^{-2/3 +\epsilon}},\qquad \E \widetilde f_N = \Omega_f +\landauO{N^{-2/3+\epsilon}},
\end{equation}
 and \[ \sqrt N\left[f_N-\Omega_f\right]-\xi_{11}\int_{-2}^2 f'(x)\rho(x)\diff x \Rightarrow\Delta_f \quad\text{and}\quad \sqrt N\left[\widetilde f_N-\Omega_f\right]-\xi_{11}\int_{-2}^2 \frac{xf''(x)}{2}\rho(x)\diff x \Rightarrow \Delta_f\] converge in distribution to $\Delta_f$. Any fixed moment of these random variables converges at least at a rate of $\landauO{N^{-1/6+\epsilon}}$ to the corresponding Gaussian moments. 
\end{theorem}
\begin{remark}
Theorem \ref{thm:mainThm} shows that the fluctuations of $f_N$ and $\widetilde f_N$ are always Gaussian if $\int f'(x)\rho(x)\diff x=0$ or $\int f''(x) x\rho(x)\diff x=0$, respectively. For generic $f$ not fulfilling these conditions the fluctuations are Gaussian if and only if $h_{11}$ follows a Gaussian distribution. 
\end{remark}

By polarization identity, the limiting covariances of $\sqrt{N}\big[ f_N- \Omega_f\big]$ and $\sqrt{N}\big[ g_N- \Omega_g\big]$
may be obtained for any pair of functions $f,g $. In particular, Theorem~\ref{thm:mainThm} extends to complex test functions $f$
by considering its real and imaginary parts separately.  We also note that the condition $f\in H^2$ is not essential. 
The theorem holds for any $f\in H^1$, provided that $\int_{-2}^2\abs{\rho'(x) x f'(x)}\diff x<\infty$.
 Finally, we remark that the same statement holds for generalized Wigner matrices where we assume $s_{ij}=1$ only for $i=1$ and $j>1$. For $i\geq 2$ we only need to assume \begin{align}\sum_{j\geq 2}s_{ij}=N-1\label{eq:sij},\qquad\max_{i,j}s_{ij}\leq C\end{align} for some constant $C$. We leave it to the reader to check that our proof carries over with minor modifications to this general case, as well.  

 Applied to rectangular Young diagrams, this result translates to:
\begin{theorem} Let the Wigner matrix $H$ satisfy \eqref{moments},  $s_{ij}=1$ for $(i,j)\not=(1,1)$, $\E\abs{h_{1j}}^4=\sigma_4/N^2$ for $j=2,\dots,N$ and $\E h_{ij}^2=\sigma_2/N$ for $i<j$.  Then -- in the sense of Theorem \ref{thm:mainThm} and with the same error bounds -- we  asymptotically have 
\[w_N(E)\approx\Omega(E) + \frac{1}{\sqrt N}\left[ \widehat\Delta(E)+ \xi_{11}\frac{E\sqrt{(4-E^2)_+}+4\arcsin E/2}{2\pi} \right]\]and \[ \widetilde w_N(E)\approx\Omega(E) + \frac{1}{\sqrt N}\left[ \widehat\Delta(E)+ \xi_{11}\frac{E\sqrt{(4-E^2)_+}}{2\pi}\right],\] where $\widehat\Delta(E)$ is a centered Gaussian,
independent of $\xi_{11}$.  Its  variance  is given by the explicit formulas  
 \begin{align*}\E [\widehat\Delta(E)]^2&=V(E)\defeq V_1(E)+\abs{\sigma_2}^2 V_{\sigma_2}(E)+(\sigma_4-1)V_2(E),\\ V_1(E)&=1 - \frac{(4-E^2)_+^3}{9\pi^2} - \frac{\left(E\sqrt{(4-E^2)_+}+4\arcsin E/2\right)^2}{4\pi^2}, \quad V_2(E)=\frac{(4-E^2)_+^3}{9\pi^2},
\end{align*} where it is understood that $\arcsin(\pm x)=\pm\pi/2$ for $x>2$. The correction term $V_{\sigma_2}(E)$, that
 is only needed when $\sigma_2\not=0$, can be obtained via the general formula for $V_{\sigma_2}$ from \eqref{eq:Fsigma2}. For the special case of real symmetric $H$, we have $V_{\sigma_2}(E)=V_1(E)$. 
\end{theorem} 
A simple inspection also shows that $\widetilde w_N$ not only becomes deterministic for $\abs{E}\ge 2$, but it has smaller fluctuation  than $w_N$ everywhere. Furthermore, both $w_N$ and $\widetilde w_N$ have fluctuations precisely of order $N^{-1/2}$ in $E\in[-2,2]$, while outside of this interval only $w_N$ has fluctuations of precisely order $N^{-1/2}$ and $\widetilde w_N$ has strictly smaller fluctuations.

\section{Variance Computation}
In this section we prove Theorem \ref{thm:mainThm} in the sense of mean and variance. The proof for higher moments and the convergence of distribution will be given in Section \ref{sec:moments}.  We first introduce a commonly used (see, e.g., \cite{2012arXiv1212.0164E}) notion of high-probability bound which helps in keeping the notation compact. 
\begin{definition}[Stochastic Domination]\label{def:stochDom}
If \[X=\left( X^{(N)}(u) \,\lvert\, N\in\N, u\in U^{(N)} \right)\quad\text{and}\quad Y=\left( Y^{(N)}(u) \,\lvert\, N\in\N, u\in U^{(N)} \right)\] are families of random variables indexed by $N$, and possibly some parameter $u$, then we say that $X$ is stochastically dominated by $Y$, if for all $\epsilon, D>0$ we have \[\sup_{u\in U^{(N)}} \P\left[X^{(N)}(u)>N^\epsilon  Y^{(N)}(u)\right]\leq N^{-D}\] for large enough $N\geq N_0(\epsilon,D)$. In this case we use the notation $X\prec Y$. Moreover, if we have $\abs{X}\prec Y$, we also write $X=\stO{Y}$.
\end{definition}
It can be checked (see \cite[Lemma 4.4]{2012arXiv1212.0164E}) that $\prec$ satisfies the usual arithmetic properties, e.g. 
if $X_1\prec Y_1$ and $X_2\prec Y_2$, then also  $X_1+X_2\prec Y_1 +Y_2$ and  $X_1X_2\prec Y_1 Y_2$.
We will say that a (sequence of) events  $A=A^{(N)}$  holds with \emph{overwhelming probability} if $\P (A^{(N)}) \ge 1- N^{-D}$ for any $D>0$ and $N\ge N_0(D)$.  In particular, under the conditions \eqref{moments}, we have $h_{ij}\prec N^{-1/2}$ and  $\max_k\abs{\lambda_k} \le 3$ with overwhelming
probability.

Let $\chi\colon\R\to\R$ be a smooth cut-off function which is constant $1$ inside $[-5,5]$ and constant $0$ outside $[-10,10]$. Now define \[f_\chi(x)\defeq f(x)\chi(x)\] and its almost analytic extension \[f_\C(x+i\eta)\defeq \left[f_\chi(x)+i\eta f'_\chi(x)\right]\chi(\eta).\] 
Clearly, $f_\C$ is bounded and compactly supported. Then, \begin{align*}\partial_{\bar z}f_\C(x+i\eta)&=\frac{1}{2}\frac{\partial}{\partial x}f_\C(x+i\eta)+\frac{i}{2}\frac{\partial}{\partial \eta}f_\C(x+i\eta)=\frac{i\eta}{2}\chi(\eta) f_\chi''(x)+\frac{i}{2}\chi'(\eta)\left[ f_\chi(x)+i\eta  f_\chi'(x)\right]
\end{align*}
 and we note that for small $\eta$, \begin{align}\label{fCestimate}\partial_{\bar z}f_\C(x+i\eta)=\landauO{\eta} \quad\text{and} \quad\partial_\eta\partial_{\bar z}f_\C(x+i\eta) =\landauO{1}.\end{align}  For real $\lambda$ we have \[f_\chi(\lambda)=\frac{1}{2i\pi}\int_{\C} \frac{\partial_{\bar z} f_\C(z)}{\lambda-z}\diff \bar z\wedge\diff z=\frac{1}{\pi}\int_{\R^2} \frac{\partial_{\bar z} f_\C(x+i\eta)}{\lambda-x-i\eta}\diff x \diff \eta\] whenever $f\in C^2(\R)$, as follows from Cauchy's Theorem. Since the left hand side of this equality is real, it suffices to integrate the real part of the integrand on the right hand side which conveniently is symmetric with respect to the real axis. Consequently, \begin{align}f_\chi(\lambda)=\frac{2}{\pi}\Re\int_\R\int_{\R_+}\frac{\partial_{\bar z} f_\C(x+i\eta)}{\lambda-x-i\eta}\diff \eta \diff x.\label{eq:HelfferSjostrand}\end{align} Eq.~\eqref{eq:HelfferSjostrand} is commonly known as the \emph{Helffer-Sj\"{o}strand formula} \cite{Helffer1989}. One can easily check that eq.~\eqref{eq:HelfferSjostrand} extends to $H^2(\R)$ functions.  The cut-off was chosen in such a way that with overwhelming probability $f(\lambda_k)=f_\C(\lambda_k)$ and $f(\lambda_k-h_{11})=f_\C(\lambda_k-h_{11})$ and therefore eq.~\eqref{eq:HelfferSjostrand} yields 
\begin{align}f_N&=\frac{2}{\pi}\Re\int_{\R}\int_{\R_+}\partial_{\bar z}f_\C(x+i\eta)\left[\Tr G(x+i\eta)-\Tr \widehat G(x+i\eta)\right]\diff \eta\diff x\label{eq:fN}\end{align} and \begin{align}\widetilde f_N&=\frac{2}{\pi}\Re\int_{\R}\int_{\R_+}\partial_{\bar z}f_\C(x+i\eta)\left[\Tr G(x+h_{11}+i\eta)-\Tr \widehat G(x+h_{11}+i\eta)\right]\diff \eta\diff x\label{eq:tildefN},\end{align}
where for convenience we defined \[H = \left( \begin{matrix}
h_{11}& h^*\\ h & \widehat H
\end{matrix} \right),\quad  H^{(1)}= \left( \begin{matrix}0 & 0\\0 & \widehat H
\end{matrix} \right), \quad G(z)\defeq(H-z )^{-1},\quad  \widehat G(z)\defeq (\widehat H-z)^{-1}, \quad G^{(1)}(z)\defeq (H^{(1)}-z)^{-1}.\] We also introduce the short hand notations \[\Delta_N(z)\defeq \Tr G(z)-\Tr \widehat G(z)\quad \text{and}\quad \widetilde\Delta_N(z)\defeq \Tr G(z+h_{11})-\Tr \widehat G(z+h_{11}).\]
From the Schur complement formula we find
\begin{align}\label{eq:mNtilde}
\Delta_N(z)=\frac{1+\braket{h,\widehat G(z)^2h}}{h_{11}-z-\braket{h,\widehat G(z)h}} \quad\text{and}\quad \widetilde\Delta_N(z)=\frac{1+\braket{h,\widehat G(z+h_{11})^2h}}{-z-\braket{h,\widehat G(z+h_{11})h}}.
\end{align}
The basic strategy now is that we identify the leading order behavior of these two expressions and then handle the fluctuations separately. To do so, we firstly exclude a critical area very close to the real line. Since \[\abs{\eta+\eta\braket{h,\widehat G(x+i\eta)^2h}}\leq \eta+\Im\braket{h,\widehat G(x+i\eta)h}\leq\abs{x_0+x+i\eta+\braket{h,\widehat G(x+i\eta)h}}\] for any $x_0\in\R$ we find that \[\abs{\eta \, \Delta_N(x+i\eta)}\leq 1\quad\text{and}\quad \abs{\eta \, \widetilde\Delta_N(x+i\eta)}\leq 1\] for all $\eta>0$. Therefore we can restrict our integrations in \eqref{eq:fN}--\eqref{eq:tildefN} 
 to the domain $\Im z>\eta_0\defeq N^{-2/3}$ and find that \[f_N=\frac{2}{\pi}\Re\int_{\R}\int_{\eta_0}^{10} \partial_{\bar z}f_\C(x+i\eta)\Delta_N(x+i\eta)\diff \eta\diff x+\stO{N^{-2/3}}\] and \[\widetilde f_N=\frac{2}{\pi}\Re\int_{\R}\int_{\eta_0}^{10} \partial_{\bar z}f_\C(x+i\eta)\widetilde\Delta_N(x+i\eta)\diff \eta\diff x+\stO{N^{-2/3}}.\] 
For $\Im z=\eta\geq\eta_0$ we claim that the leading order of $\Delta_N$ and $\widetilde\Delta_N$ is given by 
\begin{align}\label{eq:hatDelta}
\widehat\Delta_N (z)\defeq \frac{1+\frac{1}{N}\Tr \widehat G(z)^2}{-z-\frac{1}{N}\Tr \widehat G(z)}. 
\end{align} 
Accordingly, we split the proof effectively into two parts. We define \begin{align}
\widehat\Omega_f \defeq \frac{2}{\pi}\Re\int_{\R}\int_{\eta_0}^{10} \partial_{\bar z}f_\C(x+i\eta)\widehat\Delta_N (x+i\eta)\diff \eta\diff x \label{eq:OmegaF}
\end{align}
 and \begin{align} \label{eq:FN}
F_N &\defeq \frac{2}{\pi}\Re\int_{\R}\int_{\eta_0}^{10} \partial_{\bar z}f_\C(x+i\eta)\left[\Delta_N(x+i\eta)-\widehat\Delta_N (x+i\eta)\right]\diff \eta\diff x
\\ \widetilde F_N &\defeq \frac{2}{\pi}\Re\int_{\R}\int_{\eta_0}^{10} \partial_{\bar z}f_\C(x+i\eta)\left[\widetilde\Delta_N(x+i\eta)-\widehat\Delta_N (x+i\eta)\right]\diff \eta\diff x,\label{eq:tildeFN}\end{align} so that \[f_N=\widehat\Omega_f+F_N+\stO{N^{-2/3}}\quad\text{and}\quad \widetilde f_N=\widehat\Omega_f+\widetilde F_N+\stO{N^{-2/3}}.\]
\begin{proposition}[Leading Order]\label{prop:leading}
Under the assumptions of Theorem \ref{thm:mainThm} we have that \[ \widehat\Omega_f = \Omega_f + \stO{N^{-2/3}}.\tag*{\qedhere} 
\]
\end{proposition}
\begin{proposition}[Fluctuations]\label{prop:fluctuations}
Under the assumptions of Theorem \ref{thm:mainThm} we have
 that \[\E F_N^2= \frac{1}{N} V_f+\stO{N^{-7/6}}\quad\text{and}\quad \E \widetilde F_N^2= \frac{1}{N} \widetilde V_f+\stO{N^{-7/6}}.\tag*{\qedhere} 
\]\end{proposition}
 Note that the error terms in these propositions are deterministic and hence could also be written as $\landauO{N^{-2/3+\epsilon}}$ or $\landauO{N^{-7/6+\epsilon}}$ 
for any $\epsilon>0$, respectively, but for simplicity we keep the $\stO{\ldots}$ notation for deterministic quantities as well. 
 
The positivity of $V_f$ and $\widetilde V_{f}$ defined \eqref{eq:Vf1}--\eqref{eq:Vf2} follows from $1=\sigma_2\le \sigma_4$
and  from simple Schwarz inequalities
\begin{align*}
   \left( \int_{-2}^2 x\rho(x) f'(x) \diff x\right)^2 &\le  \int_{-2}^2 \rho(x) f'(x)^2 \diff x,\\
   \left( \int_{-2}^2  x\rho(x) \left(f'(x) -\int \rho f'\right)\diff x\right)^2 &\le \left( \int_{-2}^2 \rho(x)\left(f(x) -\int \rho f'\right)^2\rho(x)\diff x\right),
 \end{align*}
using that the semicircle density $\rho$ is symmetric and $\int x^2 \rho(x) \diff x =1$.

\subsection{Leading Order Integral}
This section is devoted to the proof of Proposition \ref{prop:leading}. We rely on the local semicircle law  in the averaged form
(see \cite{Erdos20121435} or \cite[Theorem 2.3]{2012arXiv1212.0164E})
 \begin{align}\label{eq:localSC}
m_N(z)\defeq \frac{1}{N}\Tr\widehat G(z) = m(z)+\stO{\frac{1}{N\eta}}
\end{align} and the entry-wise form 
\begin{align}\label{eq:localSCentry}
\widehat G_{ij}(z)-\delta_{ij}m(z)\prec\frac{1}{\sqrt{N\eta}}
\end{align} 
 which holds true for all $\eta=\Im z>\eta_0$. Here $m(z)$ is the Stieltjes transform of the semicircle distribution, i.e.,
\[
m(z)\defeq \int_{-2}^2\frac{1}{x-z}\rho(x)\diff x=\frac{1}{2\pi}\int_{-2}^2\frac{\sqrt{4-x^2}}{x-z}\diff x=\frac{-z+\sqrt{z^2-4}}{2},
\] 
where we chose the branch of the square root with positive imaginary part. Note that  $\widehat G$ is an $(N-1)\times (N-1)$
matrix but we still normalize its trace by $1/N$ to define $m_N$; this unconventional notation will simplify some formulas later. Strictly speaking,
the sum of the variances in each row of $\widehat H$ is not exactly one as required in  
\cite{Erdos20121435, 2012arXiv1212.0164E}, partly due to the removal of one column and partly
due to the relaxed bound $\E |h_{ii}|^2\le C/N$ on the diagonal elements. Nevertheless, we still have
$\sum_{i=2}^N \E |h_{ij}|^2 = 1+ \landauO{N^{-1}}$ for each $j=2,3,\ldots, N$ and 
the proof of \cite[Theorem 2.3]{2012arXiv1212.0164E} goes through. The only small change is that the $\landauO{N^{-1}}$
error term above gives rise to an additional  term  of size $\landauO{1/N\eta}$ in the definition of $\Upsilon_i$ in (5.7)-(5.8) of \cite{2012arXiv1212.0164E}
using the trivial bound $|v_i|= |G_{ii}-m|\le 2/\eta$ with the notation of that paper. Since the error bound on $\Upsilon_i$ used
in that proof is bigger than $\landauO{1/N\eta}$, see \cite[Lemma 5.2]{2012arXiv1212.0164E}, the rest of the proof is unchanged. 

Thus 
\[
\widehat\Delta_N (z)=\frac{1+\frac{1}{N}\Tr\widehat G(z)^2}{-z-m(z)}+\stO{\frac{1}{N\eta}}=m(z)\left[1+\frac{1}{N}\Tr\widehat G(z)^2\right]+\stO{\frac{1}{N\eta}},
\] 
where we used the relation $m(z)=1/(-z-m(z))$. Since $\partial_{\bar z}f_\C(x+i\eta)=\landauO{\eta}$ 
for small $\eta$ the error term, when inserted in \eqref{eq:OmegaF} only gives a contribution of $1/N$. Thus eq.~\eqref{eq:OmegaF} becomes \begin{align*}
\widehat\Omega_f=\frac{2}{\pi}\Re\int_{\R}\int_{\eta_0}^{10} \partial_{\bar z}f_\C(z)m(z) \left[1+\frac{1}{N}\Tr \widehat G(z)^2\right]\diff \eta\diff x +\stO{N^{-1}},
\end{align*} where from now on we shall always use the shorthand notation $z=x+i\eta$. Noting that 
\[
1+\frac{1}{N}\Tr\widehat G(x+i\eta)^2=\partial_{\eta}\left[\eta-i\frac{1}{N}\Tr\widehat G(x+i\eta)\right]=\partial_{\eta}\left[\eta-im_N(x+i\eta)\right],
\] 
 and $-i\partial_\eta h(z)=\partial_z h(z)$ for analytic $h$,  we can now perform an integration by parts to find \begin{align*}
\widehat\Omega_f &= \frac{2}{\pi}\Re\int_{\R} \partial_{\bar z}f_\C(z_0)m(z_0) \left[\eta-i m_N(z_0)\right]\diff x -\frac{2}{\pi}\Re\int_{\R}\int_{\eta_0}^{10} \partial_\eta\left(\partial_{\bar z}f_\C(z)m(z)\right) \left[\eta-im_N(z)\right]\diff \eta\diff x +\stO{N^{-1}}\\ & =\frac{2}{\pi}\Re\int_{\R} \partial_{\bar z}f_\C(z_0)m(z_0) \left[\eta-im(z_0)\right]\diff x -\frac{2}{\pi}\Re\int_{\R}\int_{\eta_0}^{10} \partial_\eta\left(\partial_{\bar z}f_\C(z)m(z)\right) \left[\eta-im(z)\right]\diff \eta\diff x +\stO{N^{-1}}\\&\quad + \frac{2}{\pi}\Re\int_{\R} \partial_{\bar z}f_\C(z_0)m(z_0) i\left[m_N(z_0)-m(z_0)\right]\diff x -\frac{2}{\pi}\Re\int_{\R}\int_{\eta_0}^{10} \partial_\eta\left(\partial_{\bar z}f_\C(z)m(z)\right) i\left[m_N(z)-m(z)\right]\diff \eta\diff x \\ & = \frac{2}{\pi}\Re\int_{\R}\int_{\eta_0}^{10} \partial_{\bar z}f_\C(z)m(z) \left[1+m'(z)\right]\diff \eta\diff x + \stO{N^{-1}}+\stO{N^{-1}|\log\eta_0|} 
\end{align*} 
where we used that $\partial_{\bar z}f_\C(x+i\eta)$ scales like $\eta$ near the real axis and the local semicircle law from eq.~\eqref{eq:localSC}. For the main term we need the following simple lemma. 
\begin{lemma}\label{lemma:stokes}
 Let $\phi,\psi\colon[-10,10]\times[0,10i]\to\C$ be functions such that $\partial_{\bar z}\psi(z)\equiv0$, $\phi, \psi \in H^1$
 and $\phi$ vanishes at the left, right and top boundary of the integration region. Then  for any $\eta_0\in [0,10]$, we have 
  \[\int_{-10}^{10}\int_{\eta_0}^{10} [\partial_{\bar z}\phi(z)]\psi(z)\diff\eta\diff x=\frac{1}{2i}\int_{-10}^{10} \phi(x+i\eta_0)\psi(x+i\eta_0)\diff x , \qquad  z=x+i\eta.\tag*{\qedhere} 
 \]
\end{lemma}
\begin{proof}
This follows from the computation 
\begin{align*}
\int_{-10}^{10}\int_{\eta_0}^{10} [\partial_{\bar z}\phi(z)]\psi(z)\diff\eta\diff x &= \frac{1}{2i}\int_{-10}^{10}\int_{\eta_0}^{10} [\partial_{\bar z}\phi(z)]\psi(z)\diff \bar z\wedge\diff z = \frac{1}{2i}\int_{-10}^{10}\int_{\eta_0}^{10} \diff \left(\phi(z)\psi(z)\diff z\right)\\ & = \frac{1}{2i}\int_{-10}^{10} \phi(x+i\eta_0)\psi(x+i\eta_0)\diff x,
\end{align*} where we used Stokes' Theorem in the ultimate step.
\end{proof}  We apply this together with $\Im m(x)[1+m'(x)]=(4-x^2)^{-1/2}$  and \eqref{fCestimate} to extend the integration to the real axis  and conclude that \begin{align*}\widehat\Omega_f=
\frac{2}{\pi}\Re\int_{\R}\int_{\eta_0}^{10} \partial_{\bar z}f_\C(z)\widehat\Delta_N (z)\diff \eta\diff x =\Omega_f+\stO{N^{-2/3}}= \frac{1}{\pi}\int_{-2}^{2} \frac{f(x)}{\sqrt{4-x^2}} \diff x+\stO{N^{-2/3}},
\end{align*} completing the proof of Proposition \ref{prop:leading}.

\subsection{Fluctuation Integral}
We now turn to the proof of Proposition \ref{prop:fluctuations}. We formulate the main estimate as a lemma:
\begin{lemma}
For any $\eta>\eta_0$ we have that \begin{align}\label{eq:TildemN1}
\Delta_N(z)-\widehat\Delta_N (z)= \partial_{z}\frac{\braket{h,\widehat G(z) h}-m_N(z)-h_{11}}{-z-m_N(z)}+\stO{\frac{1}{N\eta^2}}
\end{align}
and 
\begin{align}\label{eq:TildemN2}
\widetilde\Delta_N(z)-\widehat\Delta_N (z)= \partial_{z}\frac{\braket{h,\widehat G(z+h_{11}) h}-m_N(z)}{-z-m_N(z)}+\stO{\frac{1}{N\eta^2}}.
\end{align}
\end{lemma}
\begin{proof}
This lemma  relies on the following large deviation bound (see, e.g. \cite[Theorem C.1]{2012arXiv1212.0164E})  
\begin{align}\braket{h,Ah}
=\frac{1}{N}\Tr A+\stO{\frac{1}{N}\sqrt{\Tr\abs{A}^2}}\label{eq:LargeDeviation}.
\end{align}
To prove eq.~\eqref{eq:TildemN1} we  write the difference $\Delta_N-\widehat\Delta_N$ from   \eqref{eq:mNtilde} and \eqref{eq:hatDelta} as 
\begin{align*}
\Delta_N(z)-\widehat\Delta_N (z)=\frac{\left(-z-m_N(z) \right)\left(\braket{h,\widehat G(z)^2h}-m_N'(z)\right)-\left(-1-m_N'(z \right)\left(\braket{h,\widehat G(z)h}-m_N(z)-h_{11}\right)}{\left(-z-m_N(z) \right)^2- \left(-z-m_N(z) \right) \left(\braket{h,\widehat G(z)h}-m_N(z)-h_{11}\right)}. 
\end{align*} Now it follows from eq.~\eqref{eq:LargeDeviation} and \eqref{eq:localSC} that 
\begin{align}
\braket{h,\widehat G(z)h}-m_N(z)\prec \frac{1}{N} \sqrt{\Tr\abs{\widehat G(z)}^2}
\leq \frac{1}{N}\sqrt{\frac{1}{\eta}\Im\Tr \widehat G(z)}\prec\frac{1}{\sqrt{N\eta}}
\label{eq:LargeDeviationG}
\end{align} 
and also 
\begin{align}
\braket{h,\widehat G(z)^2h}-m_N'(z)\prec \frac{1}{N}\sqrt{\Tr\abs{\widehat G(z)}^4}\leq
 \frac{1}{N\eta}\sqrt{\Tr\abs{\widehat G(z)}^2}\prec\frac{1}{\sqrt{N\eta^3}}.
 \label{eq:LargeDeviationG2}
 \end{align}
We can therefore conclude that $\Delta_N(z)-\widehat\Delta_N (z)$ can be estimated as
\begin{align*}
&\frac{\left(-z-m_N(z) \right)\left(\braket{h,\widehat G(z)^2h}-m_N'(z)\right)-\left(-1-m_N'(z \right)\left(\braket{h,\widehat G(z)h}-m_N(z)-h_{11}\right)}{\left(-z-m_N(z) \right)^2}+\stO{\frac{1}{N\eta^2}}\\ &\qquad= \partial_{z}\frac{\braket{h,\widehat G(z) h}-m_N(z)-h_{11}}{-z-m_N(z)}+\stO{\frac{1}{N\eta^2}}. 
\end{align*} 
The proof of eq.~\eqref{eq:TildemN2} is identical and shall be omitted. 
\end{proof}
We now use eq.~\eqref{eq:TildemN1} to start estimating the fluctuations $F_N$ of $f_N$ as defined in eq.~\eqref{eq:FN} via an integration by parts (with $z_0= x+i\eta_0$)
 \begin{align*}
F_N=& -\frac{2}{\pi}\Re\int_{\R}\partial_{\bar z}f_\C(z_0)i\frac{\braket{h,\widehat G(z_0) h}-m_N(z_0)-h_{11}}{-z_0-m_N(z_0)}\diff x \\&\qquad+ \frac{2}{\pi}\Re\int_{\R}\int_{\eta_0}^{10}\partial_{\eta}\partial_{\bar z}f_\C( z)i\frac{\braket{h,\widehat G( z) h}-m_N( z)-h_{11}}{- z-m_N( z)}\diff \eta\diff x+\stO{\frac{-\log\eta_0}{N}} 
\end{align*} and continue with the estimate 
\[\frac{\braket{h,\widehat G( z) h}-m_N( z)-h_{11}}{- z-m_N( z)}=\frac{\braket{h,\widehat G( z) h}-m_N( z)-h_{11}}{- z-m(z)}+\stO{\frac{1}{(N\eta)^{3/2}}} \] 
from \eqref{eq:LargeDeviationG} and  \eqref{eq:localSC} 
to find that \begin{align}\nonumber
F_N=& -\frac{2}{\pi}\Re\int_{\R}m(z_0)\partial_{\bar z}f_\C(z_0)i\left[\braket{h,\widehat G(z_0) h}-m_N(z_0)-h_{11}\right]\diff x \\\nonumber&\qquad+ \frac{2}{\pi}\Re\int_{\R}\int_{\eta_0}^{10}m(z)\partial_{\eta}\partial_{\bar z}f_\C( z)i\left[\braket{h,\widehat G( z) h}-m_N( z)-h_{11}\right]\diff \eta\diff x+\stO{N^{-2/3}}\\ =& -\frac{2}{\pi}\Im\int_{\R}\int_{\eta_0}^{10}m(z)\partial_{\eta}\partial_{\bar z}f_\C( z)\left[\braket{h,\widehat G( z) h}-m_N( z)-h_{11}\right]\diff \eta\diff x+\stO{N^{-2/3}}, \label{eq:Fexpansion}
\end{align} where we used in the last step that \[\abs{\partial_{\bar z}f_\C(z_0)\left[\braket{h,\widehat G(z_0) h}-m_N(z_0)-h_{11}\right]}\prec \sqrt{\frac{\eta_0}{N}}\leq N^{-2/3}
\]
 from \eqref{eq:LargeDeviationG} and \eqref{fCestimate}. Similarly one finds that \begin{align}\nonumber\widetilde F_N\defeq& \frac{2}{\pi}\Re\int_{\R}\int_{\eta_0}^{10} \partial_{\bar z}f_\C(z)[\widetilde\Delta_N(z)-\widehat\Delta_N (z)]\diff \eta\diff x\\\nonumber=&\frac{2}{\pi}\Re\int_{\R}\int_{\eta_0}^{10}m(z)\partial_{\eta}\partial_{\bar z}f_\C( z)i\left[\braket{h,\widehat G( z+h_{11}) h}-m_N(z)\right]\diff \eta\diff x+\stO{N^{-2/3}}\\\nonumber=&\frac{2}{\pi}\Re\int_{\R}\int_{\eta_0}^{10}m(z)\partial_{\eta}\partial_{\bar z}f_\C( z)i\left[\braket{h,\widehat G( z+h_{11}) h}-m_N(z+h_{11})+m(z+h_{11})-m(z)\right]\diff \eta\diff x+\stO{N^{-2/3}}\\=& -\frac{2}{\pi}\Im\int_{\R}\int_{\eta_0}^{10}m(z)\partial_{\eta}\partial_{\bar z}f_\C( z)\left[\braket{h,\widehat G( z+h_{11}) h}-m_N(z+h_{11})+h_{11}m'(z)\right]\diff \eta\diff x+\stO{N^{-2/3}}\label{eq:Ftildeexpansion}
\end{align} where in the penultimate step we used the local semicircle law \eqref{eq:localSC} and integrated the error term $(N\eta)^{-1}$ at an expense of $N^{-1}|\log\eta_0|$ and in the last step estimated \[m(z+h_{11})-m(z)=h_{11}m'(z)+\stO{\frac{1}{\eta^{3/2}N}},\]  where the error term, after integration,  contributes an error of at most $N^{-2/3}$. 

Both fluctuation estimates from eqs. \eqref{eq:Fexpansion} and \eqref{eq:Ftildeexpansion} have two convenient properties: Firstly,  the leading order expressions for $F_N$ and $\widetilde F_N$ have zero mean and secondly, the fluctuations in them stemming from $h_{11}$ and the ones from $h$ and $\widehat G(z)$ can be separated.  Indeed, \[\E\left[\braket{h,\widehat G( z+h_{11}) h}-m_N(z+h_{11})+h_{11}m'(z)\right]^2=\E\left[\braket{h,\widehat G( z+h_{11}) h}-m_N(z+h_{11})\right]^2+\E\left[h_{11}m'(z)\right]^2\]
since the expectation with respect to $h$, conditioned on $h_{11}$ of the first term on the rhs.~is zero and $h$ and $h_{11}$ are independent. Similarly, \[\E\left[\braket{h,\widehat G( z) h}-m_N(z)-h_{11}\right]^2=\E\left[\braket{h,\widehat G( z) h}-m_N(z)\right]^2+\E\left[h_{11}\right]^2.\]
Therefore we can start computing the variances as \begin{align}\nonumber
\E F_N^2 &= \E\left(\frac{2}{\pi}\Im\int_{\R}\int_{\eta_0}^{10}m(z)\partial_{\eta}\partial_{\bar z}f_\C( z)\left[\braket{h,\widehat G( z) h}-m_N( z)\right]\diff \eta\diff x\right)^2 \\&\qquad\qquad+\frac{s_{11}}{N}\left(\frac{2}{\pi}\Im\int_{\R}\int_{0}^{10}m(z)\partial_{\eta}\partial_{\bar z}f_\C( z)\diff \eta\diff x\right)^2+\stO{N^{-7/6}}\label{eq:EF2}
\end{align} 
and
\begin{align}\nonumber
\E \widetilde F_N^2 &=\E\left(\frac{2}{\pi}\Im\int_{\R}\int_{\eta_0}^{10}m(z)\partial_{\eta}\partial_{\bar z}f_\C( z)\left[\braket{h,\widehat G( z+h_{11}) h}-m_N( z+h_{11})\right]\diff \eta\diff x\right)^2 \\&\qquad\qquad+\frac{s_{11}}{N}\left(\frac{2}{\pi}\Im\int_{\R}\int_{0}^{10}m(z)m'(z)\partial_{\eta}\partial_{\bar z}f_\C( z)\diff \eta\diff x\right)^2+\stO{N^{-7/6}}.\label{eq:EtildeF2}
\end{align} 
Note that in the second terms we extended the integration domain of $\eta$ starting from 0 instead of $\eta_0$
at a negligible error. The second terms are already deterministic and explicitly computable using Lemma \ref{lemma:stokes} and they give rise to the integral coefficients in \eqref{eq:mainEQ}.  When taking expectations, we frequently use the property that
if $X=\stO{Y}$, $Y\ge 0$ and
$\abs{X}\le N^C$ for some constant $C$, then $\E \abs{X}\prec \E Y$, or, equivalently, $\E \abs{X}\le N^\epsilon \E Y$ for any $\epsilon>0$ and $N\ge N_0(\epsilon)$. 

 For the first term we introduce short-hand notations \begin{align}g(z)\defeq \frac{2}{\pi} m(z)\partial_{\eta}\partial_{\bar z} f_\C(z), \qquad X(z)\defeq \sqrt{N}\left[\braket{h,\widehat G( z) h}-m_N(z)\right]\label{def:gX}\end{align}
to write \[F_N'\defeq \frac{1}{\sqrt N}\E\left(\Im\int_{\R}\int_{\eta_0}^{10}g(z)X(z)\diff \eta\diff x\right)^2,\qquad \widetilde F_N'\defeq \frac{1}{\sqrt{N}}\E\left(\Im\int_{\R}\int_{\eta_0}^{10}g(z)X(z+h_{11})\diff \eta\diff x\right)^2.\]
For complex numbers $z,w$ we can expand 
\begin{align}(\Im z)(\Im w)
=\frac{1}{2}\Re\left[ \bar z w-z w\right]
\label{eq:Im2}
\end{align}
 to write out \begin{align}F_N'&=\frac{1}{N}\frac{1}{2}\Re\iint_\R\iint_{\eta_0}^{10} \left[ g(z)g(\bar z')\E X(z)X(\bar z')-g(z)g(z')\E X(z)X(z')\right]\diff\eta\diff\eta'\diff x\diff x' \label{eq:varIntegral}\end{align} where we used that $\overline{X(z)}=X(\bar z)$ and $\overline{g(z)}=g(\bar z)$. 
To work out the expectations, we expand \begin{align*}
X(z)X(z') &=  N\left(\sum_{i\not=j}\overline{h_i}G_{ij}h_j+\sum_{i}\left[\abs{h_i}^2-\frac{1}{N}\right]G_{ii}\right)\left(\sum_{l\not=k}\overline{h_l}G'_{lk}h_k+\sum_{l}\left[\abs{h_l}^2-\frac{1}{N}\right]G'_{ll}\right)
\end{align*} 
where we introduced the shorthand notations 
\[ 
G=\widehat G(z), \qquad G'=\widehat G(z'). 
\]
Note that we have redefined the notation $G$  but it should not create any confusion since  the full resolvent matrix $G(z)$ will not appear any more
in the rest of the paper. To keep the notation simple we generally index the $(N-1)\times(N-1)$ matrices $G,G'$ and the $(N-1)$ vector $h$ by integers $\{2,\dots,N\}$. In particular, all sums involving $G$ and $G'$ run from $2$ to $N$ if not stated otherwise. 
We then compute the expectation $\E_1=\E(\cdot\lvert H^{(1)})$ conditioned on $H^{(1)}$ to find 
\begin{align}\label{eq:EXX}
\E_1[X(z)X(z')]&=  N\sum_{i\not=j} \left(G_{ij}G'_{ji} \E\abs{h_i}^2\abs{h_j}^2+G_{ij}G'_{ij} \E \overline{h_i^2}h_j^2 \right)+N\sum_i \E\left[\abs{h_i}^2-\frac{1}{N}\right]^2 G_{ii}G'_{ii}\\ &= \frac{1}{N}\sum_{i\not=j}\left(G_{ij}G'_{ji}+\abs{\sigma_2}^2 G_{ij}G'_{ij} \right)+\frac{\sigma_4-1}{N}\sum_i G_{ii}G'_{ii}\nonumber \\
 &=\frac{1}{N}\sum_{i\not=j}\left(G_{ij}G'_{ji}+\abs{\sigma_2}^2 G_{ij}G'_{ij} \right)+(\sigma_4-1)m(z)m(z')+\stO{\frac{1}{\sqrt{N\eta}}+\frac{1}{\sqrt{N\eta'}}+\frac{1}{N\sqrt{\eta\eta'}}},\nonumber
\end{align} where  we recall that $\E h_{ij}^2=\sigma_2/N$ for $i<j$ and $\E h_{ij}=\overline{\sigma_2}/N$ for $i>j$. For the computation of the first term we need a lemma:
\begin{lemma}\label{lemma:trGG}
Let $\eta,\eta'>0$. Then for $z,z'$ with $\abs{\Im z}= \eta$ and $\abs{\Im z'}=\eta'$ it holds that 
\begin{align}\label{eq:trGG}
\frac{1}{N}\sum_{i\not=j}G_{ij}G'_{ji}=\frac{m(z)^2m(z')^2}{1-m(z)m(z')}+ \stO{\frac{1}{(\eta+\eta')\sqrt{N\eta\eta'}} \Big[\frac{1}{\sqrt\eta} +
\frac{1}{\sqrt{\eta'}} + \frac{1}{\sqrt{N\eta\eta'}}\Big]
} 
\end{align} 
and 
\begin{align}\label{eq:trGGt}\frac{1}{N}\sum_{i\not=j}G_{ij}G'_{ij}&=m(z)m(z')\frac{(1+m(z)m(z')\Re \sigma_2) \frac{\tan[m(z)m(z')\Im\sigma_2]}{m(z)m(z')\Im\sigma_2}-1}{1-\Re \sigma_2\frac{\tan[m(z)m(z')\Im\sigma_2]}{\Im\sigma_2}}\\ &\qquad\nonumber+
 \stO{\frac{1}{(\eta+\eta')\sqrt{N\eta\eta'}} \Big[\frac{1}{\sqrt\eta} +
\frac{1}{\sqrt{\eta'}} + \frac{1}{\sqrt{N\eta\eta'}}\Big]}
\end{align} 
(if $\Im \sigma_2=0$, then we use the convention that $\tan x/x =0$ for $x=0$). 
\end{lemma}
We remark that the $(\eta+\eta')^{-1}$ factor in the 
error term can be substantially improved if $\Im z$ and $\Im z'$ has the same sign, see e.g. \cite{2013AnHP...14.1837E} for the special $z=z'$ case, but the same argument works  in the general case.  
\begin{proof}
The proof of this lemma follows the techniques used in \cite{2013AnHP...14.1837E}. 
We let $G^{(j)}$ denote the resolvent of the minor of $\widehat H$ after removing the $j$-th row and column. We have the
resolvent identity
$$
G_{ij} = -G_{ii} \sum_j^{(i)} G_{ik}^{(j)}h_{kj}, \qquad i\ne j,
$$
where the summation runs over all $j=2,3, \ldots, N$ except $j=i$; this exclusion is indicated with the upper index on the summation.
Using the local semicircle law \eqref{eq:localSC}, we find that for any fixed $i$ 
\begin{align}\label{eq:firstEqLemma}\frac{1}{N}\sum_{j}^{(i)} \E_j [G_{ij} G'_{ji}]&=\frac{1}{N}\sum_{j}^{(i)} \E_j \left[\frac{m(z)m(z')}{G_{jj}G'_{jj}}G_{ij} G'_{ji}\right]+\stO{\Psi}\\
\nonumber &=\frac{1}{N}m(z)m(z')\sum_{j}^{(i)}\E_j\left[\left(\sum_{k}^{(j)}G^{(j)}_{ik}h_{kj}\right)\left(\sum_{l}^{(j)} h_{jl} G'^{(j)}_{li} \right)\right]+\stO{\Psi}\\
\nonumber &=\frac{1}{N^2}m(z)m(z')\sum_{j}^{(i)}\sum_{k}^{(j)}G^{(j)}_{ik}G'^{(j)}_{ki}+\stO{\Psi}\\
\nonumber &=\frac{1}{N^2}m(z)m(z')\sum_{j}^{(i)}\left[\sum_{k}^{(ij)}G_{ik}G'_{ki}+G_{ii}G'_{ii}\right]+\stO{\Psi}\\
\nonumber &=\frac{1}{N}m(z)m(z')\left[\sum_{k}^{(i)}G_{ik}G'_{ki}+m(z)m(z')\right]+\stO{\Psi}\end{align} where in the fourth equality we used \[G_{ik}^{(j)}=G_{ik}-\frac{G_{ij}G_{jk}}{G_{jj}}=G_{ik}+\stO{\frac{1}{N\eta}}\] and the analogous identity for $G'$ and we introduced the short hand notation \[\Psi=\frac{1}{\sqrt{N^3\eta^2\eta'}}+\frac{1}{\sqrt{N^3\eta\eta'^2}}+\frac{1}{N^2\eta\eta'}\] for the error term. 
We now follow the fluctuation averaging analysis from \cite[Proof of Prop.~5.3 in Sections 6--7]{2013AnHP...14.1837E}.
This proof was  given for the case when the spectral parameters of the resolvents were identical, $z=z'$,
but a simple inspection shows  that the argument verbatim also applies to the $z\not=z'$ case. We conclude that 
\begin{equation}\label{eq:FluctAv}
\frac{1}{N}\sum_{j}^{(i)}G_{ij}G'_{ji}=\frac{1}{N}\sum_{j}^{(i)}\E_j[G_{ij}G'_{ji}]+\stO{\Psi}.
\end{equation} 
 Therefore, after summing over $i$ we have \begin{align}\label{eq:firstConclusion}\left[1-m(z)m(z')\right]\frac{1}{N}\sum_{j\not= i}G_{ij}G'_{ji}=m(z)^2m(z')^2+\stO{N\Psi}.\end{align} To finish the proof, we note that by an elementary calculation 
 \begin{equation}\label{eq:etaeta}
 \frac{1}{\abs{1-m(z)m(z')}}\le\frac{C}{\eta+\eta'} 
 \end{equation} since \begin{align}\abs{m(x+i\eta)}\leq 1-c\abs{\eta} \label{eq:mEstimate}\end{align}
  and therefore \[\frac{1}{N}\sum_{i\not=j} G_{ij}G'_{ji}=\frac{m(z)^2m(z')^2}{1-m(z)m(z')}+\stO{\frac{N\Psi}{\eta+\eta'}}.\]
  This completes the proof of \eqref{eq:trGG}.
 
 For the proof of eq.~\eqref{eq:trGGt} we have to derive a vector self-consistent equation instead of the scalar one. We again start by noting that for $i\not=j$ 
 \begin{align*}\E_j G_{ij}G'_{ij}&=m(z)m(z') \sum_k^{(j)} G_{ik}G'_{ik}\E h_{kj}^2+\stO{\Psi}\\
 &=m(z)m(z')\sum_k^{(i)} G_{ik}G'_{ik}\E h_{kj}^2+m(z)^2m(z')^2\E h_{ij}^2+\stO{\Psi}\\
 &=m(z)m(z')\sum_k^{(i)} F_{jk}  \E_k G_{ik}G'_{ik} +m(z)^2m(z')^2 F_{ji}+\stO{\Psi},\end{align*} where we introduced 
 the matrix $F$ with matrix elements 
 \[
 F_{jk}\defeq\E h_{kj}^2= \frac{1}{N}\Big[ \1(k<j)\sigma_2+\1(k>j)\overline{\sigma_2}+\1(k=j)\Big].
 \]  
 For every fixed $i$, 
 we have therefore derived a self-consistent equation for the (column) vector \[v^{(i)}=\left((1-\delta_{ij})\E_j G_{ij}G'_{ij}\right)_{j=2}^{N}\] 
 which can be written as 
  \[\left[\1-m(z)m(z')F\right]v^{(i)}= m(z)^2m(z')^2 \left[F-\frac{1}{N}\1\right]e_i+\stO{\Psi},\] 
 where 
   $e_i = (0,0,\ldots 1, \ldots 0)^T$ is the standard $i$-th basis vector of $\C^{N-1}$. 
   To invert this equation while controlling the error term, we have estimate \[ \norm{\left[\1-m(z)m(z') F\right]^{-1}}_{\ell^\infty\to\ell^\infty}.\] 
   To do so, we first note that \[\norm{\left[\1-m(z)m(z') F\right]^{-1}}_{\ell^2\to\ell^2}
   \leq \left(1-\abs{m(z)}\abs{m(z')}\norm{F}_{\ell^2\to\ell^2}\right)^{ -1}
   \leq \frac{C}{\eta+\eta'}\;,\] where we used that $F$ is Hermitian and of norm at most $1$ and \eqref{eq:mEstimate}
    (the norm here is induced by the usual $\ell^2$ norm $\norm{u}_2 \defeq (\sum_i \abs{u_i}^2)^{1/2}$ on $\C^{N-1}$). 
   Next, if $(1-m(z)m(z')F)u=v$, then 
   \[
   \norm{u}_\infty\leq \norm{v}_\infty+\norm{Fu}_\infty\leq \frac{C}{\eta+\eta'}\norm{v}_\infty,
  \]
  where we used 
   \begin{align*}
  \norm{Fu}_\infty\leq \frac{1}{N}\norm{u}_1\leq \frac{1}{\sqrt{N}}\norm{u}_2 = \frac{1}{\sqrt{N}}\norm{\left[\1-m(z)m(z') F\right]^{-1}v}_2 \leq
  \frac{C}{\eta+\eta'}\frac{1}{\sqrt{N}}\norm{v}_2\leq \frac{C}{\eta+\eta'}\norm{v}_\infty,
 \end{align*}
  so that also \[\norm{\left[\1-m(z)m(z') F\right]^{-1}}_{\ell^\infty\to\ell^\infty}\leq\frac{C}{\eta+\eta'}
 \]
  for $\eta,\eta'\le C$.  After inversion we find that \[v^{(i)}=m(z)^2m(z')^2 \left(\1-m(z)m(z')F\right)^{-1}\left(F-\frac{1}{N}\1\right)e_i+\stO{\frac{\Psi}{\eta+\eta'}}.\]
 Using fluctuation averaging once more (see \eqref{eq:FluctAv}) we can conclude that
 \begin{align}\label{eq:GG'}
\frac{1}{N}\sum_{i\not=j}G_{ij}G'_{ij}&=\frac{1}{N}\sum_{i\not=j}\E_j G_{ij}G'_{ij}+\stO{N\Psi}\\ 
& = m(z)^2m(z')^2e^T \left(\1-m(z)m(z')F\right)^{-1}\left(F-\frac{1}{N}\1\right)e+\stO{\frac{N\Psi}{\eta+\eta'}}, \nonumber
\end{align} where $e= N^{-1/2} (1,\dots,1)^T \in \C^{N-1}  $. 
We now introduce the  $(N-1)\times(N-1)$  matrix \[S\defeq \frac{1}{N} \left( \begin{matrix}
0 & 1 &\dots &1\\
-1 & \ddots & \ddots & \vdots\\
\vdots & \ddots & \ddots  & 1 \\
-1 &\dots&-1& 0
\end{matrix}  \right).\] 
 Notice that $F = \frac{1}{N}\1+(\Re \sigma_2) (ee^T-\frac{1}{N}\1) + i ( \Im \sigma_2) S$. We find, through an elementary computation, that 
\begin{align*}
&m(z)m(z') \braket{ e,  \left(\1-m(z)m(z')F\right)^{-1}\left(F-\frac{1}{N}\1\right)e}
\\&\qquad\qquad=\frac{(1+m(z)m(z')\Re \sigma_2) \braket{e, \left(\1-im(z)m(z')\Im \sigma_2 S\right)^{-1}e}-1}{1-m(z)m(z')\Re \sigma_2\braket{e, \left(\1-im(z)m(z')\Im \sigma_2 S\right)^{-1}e}}+\stO{N^{-1}}.
\end{align*} 
It remains to compute 
\[\braket{e,(\1-\alpha S)^{-1}e}=\sum_{k=0}^\infty \alpha^k\braket{e,S^ke} \] for $\alpha\in\C$ with $\abs{\alpha}<1$. For any vector $f\in \C^{N-1}$, 
\[ (Sf)_n = -\frac{1}{N}\sum_{n'<n}f_{n'}+\frac{1}{N}\sum_{n'>n}f_{n'}=\frac{1}{N}\sum_{n'=2}^N h_{n-n'}f_{n'}\] 
where $ h_k\defeq \1(k<0)-\1(k>0)$. Therefore 
 \[\braket{e,S^k e}=N^{-1/2}\sum_{n=2}^N (S^ke)_n=N^{-3/2}\sum_{n,n'=2}^N h_{n-n'} (S^{k-1}e)_{n'}=\dots=N^{-k-1}\sum_{n_0,\dots,n_k=2}^N h_{n_0-n_1}\dots h_{n_{k-1}-n_k}.\] By symmetry, $\braket{e,S^ke}=0$ for odd $k$. 
 Otherwise one finds via a Riemann sum approximation that 
 \[\braket{e,S^{2k}e}=\int_{0}^1\dots\int_0^1 h(x_0-x_1)\dots h(x_{2k-1}-x_{2k})\diff x_0\dots\diff x_{2k}+\landauO{N^{-1}},\]
 where $h(x)=\1(x<0)-\1(x>0)$ is the Heaviside function and where we added the missing $n_i=1$ terms at an expense of $\landauO{N^{-1}}$. Via an easy induction we see
   that \[\braket{e,S^{2k}e}=(-1)^k \frac{2^{2k}(2^{2k}-1)}{(2k)!}B_{2k}+\landauO{N^{-1}},\] where $B_k$ is the $k$-th Bernoulli number. Consequently, 
\[\braket{e,(\1-\alpha S)^{-1}e}=\frac{\tanh\alpha}{\alpha}+\landauO{N^{-1}}.\] We now use this with $\alpha=i m(z)m(z')\Im\sigma_2$ to conclude that 
\begin{align*}m(z)m(z')\braket{e, \left(\1-m(z)m(z')F\right)^{-1}\left(F-\frac{1}{N}\1\right)e}&= \frac{(1+m(z)m(z')\Re \sigma_2) \frac{\tan[m(z)m(z')\Im\sigma_2]}{m(z)m(z')\Im\sigma_2}-1}{1-m(z)m(z')\Re \sigma_2\frac{\tan[m(z)m(z')\Im\sigma_2]}{m(z)m(z')\Im\sigma_2}}+\stO{N^{-1}}.\end{align*} 
 Combining this with \eqref{eq:GG'}, we obtain  \begin{align*}
\frac{1}{N}\sum_{i\not=j}G_{ij}G'_{ij} =& m(z)m(z')\frac{(1+m(z)m(z')\Re \sigma_2) \frac{\tan[m(z)m(z')\Im\sigma_2]}{m(z)m(z')\Im\sigma_2}-1}{1-\Re \sigma_2\frac{\tan[m(z)m(z')\Im\sigma_2]}{\Im\sigma_2}}+\stO{\frac{N\Psi}{\eta+\eta'}}.
\end{align*}
We note that, in general, this is a finite expression since $\abs{\Re \sigma_2}\leq \sqrt{1-(\Im\sigma_2)^2}$ and thus in the non-trivial case where $\Re\sigma_2\not=0$ and $\Im\sigma_2\not=0$, \[\abs{\Re\sigma_2\frac{\tan[m(z)m(z')\Im\sigma_2]}{\Im\sigma_2}}<\abs{\sqrt{1-(\Im\sigma_2)^2}\frac{\tan[\Im\sigma_2]}{\Im\sigma_2}}\leq 1.\tag*{\qedhere} 
\]
\end{proof}
We readily check that integrating the error terms in \eqref{eq:varIntegral} from \eqref{eq:EXX} and Lemma \ref{lemma:trGG} only contributes an error of magnitude $N^{-7/6}$ and conclude that if $\sigma_2=0$, then 
\begin{align}\label{eq:F1comp}F_N'&=\frac{1}{2N}\Re\iint_{-10}^{10}\iint_{\eta_0}^{10} g(z)g(\bar z')\left[\frac{m(z)^2m(\bar z')^2}{1-m(z)m(\bar z')}+(\sigma_4-1)m(z)m(\bar z') \right]\\&\quad -g(z)g(z')\left[\frac{m(z)^2m( z')^2}{1-m(z)m( z')}+(\sigma_4-1)m(z)m( z') \right]\diff\eta\diff\eta'\diff x\diff x' +  \stO{N^{-7/6}} \nonumber\\ &=\nonumber \frac{1}{2N}\Re\iint_{-10}^{10}\iint_{\eta_0}^{10} g(z)g(\bar z')\left[\sum_{k=2}^\infty [m(z)m(\bar z')]^k+(\sigma_4-1)m(z)m(\bar z') \right]\\&\quad -g(z)g( z')\left[\sum_{k=2}^\infty [m(z)m( z')]^k+(\sigma_4-1)m(z)m( z') \right]\diff\eta\diff\eta'\diff x\diff x'+\stO{N^{-7/6}}\nonumber\\ 
&= \frac{1}{N}\sum_{k=2}^\infty\left( \Im\int_{-10}^{10}\int_{\eta_0}^{10} g(z)m(z)^k\diff\eta\diff x\right)^2+\frac{\sigma_4-1}{N}\left( \Im\int_{-10}^{10}\int_{\eta_0}^{10} g(z)m(z)\diff\eta\diff x\right)^2+\stO{N^{-7/6}}\nonumber\\ 
& = \frac{1}{N}\sum_{k=2}^\infty\left( \frac{1}{\pi}\Im\int_{-10}^{10}  \frac{\partial_\eta f_\C(z_0)}{i} m(z_0)^{k+1}\diff x\right)^2+\frac{\sigma_4-1}{N}\left( \frac{1}{\pi}\Im\int_{-10}^{10}  \partial_x f_\C(z_0) m(z_0)^{2}\diff x\right)^2+\stO{N^{-7/6}}\nonumber\\
&=\frac{1}{N}\sum_{k=0}^\infty\left( \frac{1}{\pi}\Im\int_{-10}^{10}  \frac{\partial_\eta f_\C(z_0)}{i} m(z_0)^{k+1}\diff x\right)^2+\frac{\sigma_4-2}{N}\left( \frac{1}{\pi}\Im\int_{-10}^{10}  f'(x) m(x)^2\diff x\right)^{2}\nonumber\\&\qquad\qquad-\frac{1}{N}\left( \frac{1}{\pi}\Im\int_{-10}^{10}  f'(x) m(x)\diff x\right)^{2}+\stO{N^{-7/6}},\nonumber\end{align} 
where $z_0=x+i\eta_0$ and in the penultimate step we used Lemma \ref{lemma:stokes}  to write 
\begin{align*}\Im\int_{-10}^{10}\int_{\eta_0}^{10}g(z)m(z)^k\diff \eta \diff x &= \frac{2}{\pi}\Im\int_{-10}^{10}\int_{\eta_0}^{10} [\partial_{\bar z}\partial_\eta f_\C(z)]m(z)^{k+1}\diff \eta \diff x=\Im \frac{1}{i\pi}\int_{-10}^{10} \partial_\eta f_\C(z_0) m(z_0)^{k+1}\diff \eta \diff x \end{align*} and that \begin{align}\frac{\partial_\eta f_\C(z_0)}{i}=\partial_x f_\C (z_0)+\landauO{\eta_0}=f'(x)+\landauO{\eta_0}.\label{eq:almostanalytic}\end{align}

 Now that we reduced the area integral to a line integral, we go the geometric series steps backwards to further simplify the first term as
  \begin{align}\label{eq:deltaIntegral}
&\frac{1}{N}\sum_{k=0}^\infty\left( \frac{1}{\pi}\Im\int_{-10}^{10} \frac{\partial_\eta f_\C(z_0)}{i} m(z_0)^{k+1}\diff x\right)^2\\
&=\frac{1}{2N\pi^2}\Re\iint_{-10}^{10} \Bigg[\left(\frac{\partial_\eta f_\C(z_0)}{i}\right) \overline{\left(\frac{\partial_\eta f_\C(z_0')}{i}\right)} \frac{m(z_0)\overline{m(z_0')}}{1-m(z_0)\overline{m(z_0')}}-\left(\frac{\partial_\eta f_\C(z_0)}{i}\right) \left(\frac{\partial_\eta f_\C(z_0')}{i}\right)\frac{m(z_0){m(z_0')}}{1-m(z_0){m(z_0')}}\Bigg] \diff x\diff x'.\nonumber
\end{align} 
We would now like to approximate \eqref{eq:deltaIntegral} using \eqref{eq:almostanalytic}. For doing so, we have to control the error terms via the following lemma whose proof  we postpone to the end of the section.
\begin{lemma}\label{lemma:eq:errorBoundLogEta}
There exists an absolute constant $C$ such that for $z_0=x+i\eta_0$ and $z_0'=x'+i\eta_0$ with $0<\eta_0\le 1/2$ it holds that \begin{align}
 \iint_{-10}^{10} \frac{1}{\abs{1-m(z_0)\overline{m(z_0')}}}\diff x \diff x' \le C\abs{\log \eta_0} \quad\text{and}\quad \iint_{-10}^{10} \frac{1}{\abs{1-m(z_0){m(z_0')}}}\diff x \diff x' \le C\abs{\log \eta_0}.\label{eq:errorBoundLogEta}
\end{align} 
\end{lemma}
 Using Lemma \ref{lemma:eq:errorBoundLogEta} and \eqref{eq:almostanalytic} we can rewrite \eqref{eq:deltaIntegral} as 
\begin{align*}
\frac{1}{2N\pi^2}\Re\iint_\R f'(x)f'(x')\left[\frac{m(z_0)\overline{m(z_0')}}{1-m(z_0)\overline{m(z_0')}}- \frac{m(z_0){m(z_0')}}{1-m(z_0){m(z_0')}}\right]\diff x\diff x'+\landauO{\eta_0 \abs{\log\eta_0}}.
\end{align*} 
Now, an explicit computation shows \begin{align}\label{eq:VarIntFinal}
\Re\left[\frac{m(z_0)\overline{m(z_0')}}{1-m(z_0)\overline{m(z_0')}}-\frac{m(z_0)m(z_0')}{1-m(z_0)m(z_0')}\right]&=\Re\frac{-2i\Im m(z_0')}{-x-i\eta_0-2\Re m(z_0')+m(z_0)[\abs{m(z_0')}^2-1]}.
\end{align} and therefore for small $\eta_0$ and $(x,x')$ outside the square $[-2,2]^2$ the integrand of \eqref{eq:VarIntFinal} negligible. 
For $(x,x')\in[-2,2]^2$ and small $\eta_0$ we have \[\Re\frac{-2i\Im m(z_0')}{-x-i\widetilde \eta-2\Re m(z_0')+m(z_0)[\abs{m(z_0')}^2-1]}=\frac{\sqrt{4-x'^2}\eta_0}{(x-x')^2+\eta_0^2}+\landauO{\eta_0}.\] This expression acts like \[\pi\sqrt{4-x^2}\delta(x'-x)\] for small $\eta_0$. More formally, it is well known that for any $L^2$-function $h$ \[\lim_{\eta\to0}\int_{\R}\frac{\eta}{(x-x')^2+\eta^2}h(x')\diff x' = \pi h(x) \] in $L^2$-sense.
Working out an effective error term for $h\in H^1$, this
  allows us to conclude 
  \begin{align*}F_N'=\frac{1}{N}\left[\int_{-2}^{2}\rho(x)f'(x)^2\diff x-\left(\int_{-2}^2 \rho(x)f'(x)\diff x\right)^2+(\sigma_4-2)\left(\int_{-2}^2\rho(x) x f'(x)\diff x\right)^2\right]+\stO{N^{-7/6}}. 
\end{align*}
The computation for $\widetilde F_N'$ from \eqref{eq:EtildeF2}, still assuming $\sigma_2=0$, is completely analogous and there we also have \begin{align*}\widetilde F_N'=\frac{1}{N}\left[\int_{-2}^{2}\rho(x)f'(x)^2\diff x-\left(\int_{-2}^2 \rho(x)f'(x)\diff x\right)^2+(\sigma_4-2)\left(\int_{-2}^2\rho(x)xf'(x)\diff x\right)^2\right]+\stO{N^{-7/6}}. 
\end{align*} We can now conclude from eqs. \eqref{eq:EF2} and \eqref{eq:EtildeF2} that \begin{align}\label{eq:EFN}
\E F_N^2 = F_N' +\frac{s_{11}}{N}\left(\int_{-2}^2 \rho(x)f'(x)\diff x\right)^2+\stO{N^{-7/6}} = \frac{V_f}{N}+\stO{N^{-7/6}}
\end{align} 
and \begin{align}\label{eq:EFNtilde}
\E \widetilde F_N^2 = \widetilde F_N' +\frac{s_{11}}{N}\left(\int_{-2}^2\rho(x)\frac{xf''(x)}{2}\diff x\right)^2+\stO{N^{-7/6}} = \frac{\widetilde V_f}{N}+\stO{N^{-7/6}}.
\end{align}

 So far we assumed $\sigma_2=0$ in \eqref{eq:EXX}. We now consider the general case for which we need \eqref{eq:trGGt} instead of \eqref{eq:trGG}. A similar analysis shows that we have to add an additional term $\abs{\sigma_2}^2 V_{\sigma_2}$ to $V_f$ and $\widetilde V_f$  both in  \eqref{eq:EFN} and \eqref{eq:EFNtilde}, given by
\begin{align}\label{eq:Fsigma2}
V_{\sigma_2}\defeq&\frac{1}{2\pi^2}\Re\iint_\R f'(x)f'(x')\Big[m(z_0)^2\overline{m(z_0')^2}\frac{(1+m(z_0)\overline{m(z_0')}\Re \sigma_2) \frac{\tan[m(z)\overline{m(z_0')}\Im\sigma_2]}{\Im\sigma_2}-m(z)\overline{m(z_0')}}{1-\Re \sigma_2\frac{\tan[m(z_0)\overline{m(z_0')}\Im\sigma_2]}{\Im\sigma_2}}\\&\qquad\qquad\qquad\qquad\qquad - m(z_0)^2m(z_0')^2\frac{(1+m(z_0){m(z_0')}\Re \sigma_2) \frac{\tan[m(z_0){m(z_0')}\Im\sigma_2]}{\Im\sigma_2}-m(z_0){m(z_0')}}{1-\Re \sigma_2\frac{\tan[m(z_0){m(z_0')}\Im\sigma_2]}{\Im\sigma_2}}\Big]\diff x \diff x'.\nonumber
\end{align} 
For the special case $\sigma_2\in \R$ eq.~\eqref{eq:trGGt} simplifies to  \[\frac{1}{N}\sum_{i\not=j}G_{ij}G'_{ij}=\frac{m(z)^2m(z')^2\Re \sigma_2 }{1- m(z)m(z')\Re \sigma_2}+\stO{\frac{1}{(\eta+\eta')\sqrt{N\eta^2\eta'}}+\frac{1}{(\eta+\eta')\sqrt{N\eta\eta'^2}}+\frac{1}{N(\eta+\eta')\eta\eta'}}.\] In particular, for symmetric $H$, where $\sigma_2=1$ we find that eq.~\eqref{eq:Fsigma2} simplifies to $V_{\sigma_2}=V_{f,1}$. 
This completes the proof of Proposition \ref{prop:fluctuations}, modulo the proof of Lemma \ref{lemma:eq:errorBoundLogEta}.

\renewcommand*{\proofname}{Proof of Lemma \ref{lemma:eq:errorBoundLogEta}}
\begin{proof}
The proof of the second inequality is similar to the first one and will be left to the reader. For the first inequality, we split the integration in two regimes. We shall make use of the fact (see, e.g., \cite{2012arXiv1212.0164E}) that on a compact domain, say $\abs{z_0}\le 10$, we have \begin{align}
\abs{1-m(z_0)^2}\asymp \sqrt{\kappa_x+\eta}\quad\text{and}\quad \Im m(z_0)\asymp \begin{cases}
\sqrt{\kappa_x+\eta_0} &\text{if} \;\; \abs{x}\leq2,\\
\frac{\eta_0}{\sqrt{\kappa_x+\eta_0}} &\text{else},
\end{cases}\label{eq:compactSCestimate}
\end{align}
where $\kappa_x=\abs{\abs{x}-2}$ is the distance to the edge. 

Firstly in the region where $\max\{\abs{x},\abs{x'}\}\geq2$, we find \[\abs{1-m(z_0)\overline{m(z'_0)}}\geq\frac{1}{2}\left[1-\abs{m(z_0)}^2+1-\abs{m(z'_0)}^2\right]\geq c\sqrt{\kappa_{\max\{\abs{x},\abs{x'}\}}+\eta_0}, \] where  $c>0$ is a universal constant, due to the fact that $1-\abs{m(z_0)}^2=\eta_0/\Im m(z_0)$ and \eqref{eq:compactSCestimate}. 

Secondly, in the region where $\abs{x},\abs{x'}<2$, we write \[1-m(z_0)\overline{m(z'_0)}=1-\abs{m(z_0')}^2+(m(z_0')-m(z_0))\overline{m(z_0')}\] and estimate \[\abs{(m(z_0')-m(z_0))\overline{m(z_0')}}\geq c\abs{x'-x}\] 
 for some positive constant $c$. This inequality follows from writing \[\Re[m(z_0')-m(z_0)]=\int_{x}^{x'} \Re m'(u+i\eta_0)\diff u\] and 
from the estimate \[\Re m'(u+i\eta_0)=-\frac{2(\Im m(u+i\eta_0))^2}{\abs{1-m(u+i\eta_0)^2}^2}\leq -c\] for $\abs{u}\leq 2$,
where we used \eqref{eq:compactSCestimate} in the last step.
Consequently, \[\abs{1-m(z_0)\overline{m(z'_0)}}\geq c\abs{x-x'}-\abs{1-\abs{m(z_0')}^2}\geq c\abs{x-x'}-C\frac{\eta_0}{\sqrt{\kappa_{x'}+\eta_0}}\] and it follows that $\abs{1-m(z_0)\overline{m(z_0')}}\geq c\abs{x-x'}/2$ whenever $\abs{x-x'}\geq 2 (C/c) \eta_0/\sqrt{\kappa_{x'}+\eta_0}$. Together with the trivial bound $\abs{1-m(z_0)\overline{m(z_0')}}\geq c\eta$ we find that the integral in \eqref{eq:errorBoundLogEta} is bounded by $C\abs{\log\eta_0}$.
\end{proof}
\renewcommand*{\proofname}{Proof}

 \section{Computation of Higher Moments}\label{sec:moments}

We now turn to the computation of higher order moments and thereby to the completing the proof of Theorem \ref{thm:mainThm}. We recall 
from \eqref{eq:Fexpansion}--\eqref{eq:Ftildeexpansion}  that 
\[F_N=-\frac{1}{\sqrt N}\Im\int_{\R}\int_{\eta_0}^{10}g(z)\left[X(z)-\sqrt N h_{11}\right]\diff \eta\diff x+\stO{N^{-2/3}}\] and \[\widetilde F_N=-\frac{1}{\sqrt N}\Im\int_{\R}\int_{\eta_0}^{10}g(z)\left[ X(z+h_{11})+\sqrt N h_{11}m'(z)\right]\diff \eta\diff x+\stO{N^{-2/3}},
\]
 where $g$ and $X$ were defined in \eqref{def:gX}. In order to compute moments of $F_N$ and $\widetilde F_N$ we have to compute \[\E [X(z_1)\dots X(z_k)]\] for any
$k\in\N$ and $z_l\in\C\setminus\R$, $l=1,\dots,k$. We will first take the expectation with respect to the vector $h$ in the $X$'s 
which leads to a cyclic contraction of the indices of $\widehat G$. After taking the expectation
with respect to $\widehat H$, we will show that the leading order terms come
from cycles of length two.  This will  effectively show that the Wick theorem holds for the random variables $X$.
 The following lemma shows that cyclic products of at least three resolvents are in fact of lower order (the same phenomenon already was observed in \cite{2013AnHP...14.1837E}):
\begin{lemma}\label{lemma:cycles}
For closed cycles of length $k>2$ we have that \begin{align}
N^{-k/2}\sum^{\sim}_{i_1,\dots,i_k}\E G^{(1)}_{i_1i_2}\dots G^{(k-1)}_{i_{k-1}i_k}G^{(k)}_{i_ki_1}\prec \frac{1}{\left(\max_a\eta_a\right)\sqrt{N\eta_1\dots\eta_k }} \sum_{a=1}^k \frac{1}{\sqrt{\eta_a}} ,
\label{eq:HigherTraces}
\end{align} 
and for open cycles of any length $k>1$ we have that 
\begin{align}
N^{-(k+1)/2}\sum^\sim_{i_1,\dots, i_{k}}\E G^{(1)}_{i_1i_2}\dots G^{(k-1)}_{i_{k-1}i_{k}}\prec \frac{1}{\sqrt{N\eta_1\dots\eta_{k-1} }} \sum_{a=1}^{k-1} \frac{1}{\sqrt{\eta_a}} ,
\label{eq:HigherTracesFree}
\end{align} 
where $G^{(l)}\defeq \widehat G(z_l)$, $z_l\in\C\setminus\R$ with $\eta_l=\abs{\Im z_l}$ for $l=1,\dots,k$  and $\sum\limits^\sim$ indicates that the sum is performed over pairwise distinct indices. 
 Moreover, the same bound holds true when any of the $G^{(l)}$ are replaced by their transposes or Hermitian conjugates.
\end{lemma} 
\begin{proof} 
We first prove eq.~\eqref{eq:HigherTraces} and assume a real symmetric $H$. To do so, we let $\epsilon>0$ be arbitrary and will actually prove 
\[N^{-k/2}\sum^\sim_{i_1,\dots, i_k}\E G^{(1)}_{i_1i_2}\dots G^{(k-1)}_{i_{k-1}i_k}G^{(k)}_{i_ki_1}\prec \frac{N^{\epsilon}}{(\eta_1+\eta_k)\sqrt{N\eta_1\dots\eta_k }} \sum_{a=1}^k \frac{1}{\sqrt{\eta_a}}, \]
from which \eqref{eq:HigherTraces} follows due to the definition of $\prec$ in Definition \ref{def:stochDom}. We make use of the resolvent identity $G^{(1)}=\widehat H G^{(1)}/z_1-1/z_1$ to write \begin{align}
N^{-k/2}\sum^\sim_{i_1,\dots,i_k}\E G^{(1)}_{i_1i_2}\dots G^{(k)}_{i_ki_1} = \frac{1}{N^{k/2}z_1}\sum^\sim_{i_1,\dots,i_k}\sum_n \E h_{i_1n}G^{(1)}_{ni_2}G^{(2)}_{i_2i_3}\dots G^{(k)}_{i_ki_1}.\label{eq:longTraceExp}
\end{align} 
We use the standard cumulant expansion (introduced in the context of random matrices in \cite{:/content/aip/journal/jmp/37/10/10.1063/1.531589}) up to the third order term with a truncation  
\begin{equation}\label{eq:cumulant}
\E h f(h)=\E h\E f(h)+\E h^2 \E f'(h)+\landauO{ \E \abs{h^{3}\1(\abs{h}>N^{\tau-1/2})}\norm{f''}_\infty}+\landauO{\E\abs{h}^3\sup_{\abs{x}\leq N^{\tau-1/2}}\abs{f''(x)}},
\end{equation}
 where  $f$ is any smooth function of a real random variable $h$,  such that the expectations exist and $\tau>0$ is arbitrary
(for a recent similar use of this formula with truncation see \cite[Lemma 3.1]{2016arXiv160301499H}). This yields
\begin{align}\label{eq:productRule}
\E h_{i_1n}G^{(1)}_{ni_2}G^{(2)}_{i_2i_3}\dots G^{(k)}_{i_ki_1} &= \frac{1}{N}\E\frac{\partial \left[G^{(1)}_{ni_2}G^{(2)}_{i_2i_3}\dots G^{(k)}_{i_ki_1}\right]}{\partial h_{i_1n}}+R\\&=\frac{1}{N}\E\frac{\partial G^{(1)}_{ni_2}}{\partial h_{i_1n}}G^{(2)}_{i_2i_3}\dots G^{(k)}_{i_ki_1}+\frac{1}{N}\sum_{a=2}^k\E \frac{\partial G^{(a)}_{i_ai_{a+1}}}{\partial h_{i_1n}}G^{(1)}_{ni_2}\prod_{a\not= b=2}^kG^{(b)}_{i_bi_{b+1}}+R,
\nonumber\end{align} where it is understood that $i_{k+1}=i_{1}$ and $R$ is the error term resulting from the cumulant expansion. Using the identity 
\[\frac{\partial G_{ij}}{\partial h_{kl}}=-(G_{ik}G_{lj}+G_{il}G_{kj})/(1+\delta_{kl}),
\] 
and the local law \eqref{eq:localSCentry}, the first term on the rhs.~of eq.~\eqref{eq:productRule} becomes \begin{align*}
	- (G^{(1)}_{n i_1}G^{(1)}_{n i_2}+G^{(1)}_{n n}G^{(1)}_{i_1i_2}) G^{(2)}_{i_2i_3}\dots G^{(k)}_{i_ki_1}=-m(z_1) G^{(1)}_{i_1i_2}\dots G^{(k)}_{i_ki_1} + \stO{\frac{1}{N^{k/2+1/2}\sqrt{\eta\eta_1}}},
\end{align*} whenever $n\not = i_1,i_2$ and where $\eta\defeq \eta_1\dots \eta_k$. If $n=i_1$ or $n=i_2$, we shall make use of the trivial estimate 
\[- (G^{(1)}_{n i_1}G^{(1)}_{n i_2}+G^{(1)}_{n n}G^{(1)}_{i_1 i_2}) G^{(2)}_{i_2i_3}\dots G^{(k)}_{i_ki_1}\prec \frac{1}{N^{k/2}\sqrt{\eta}}.\] The $a=k$ summand of the second term in eq.~\eqref{eq:productRule} becomes \begin{align}\label{eq:ak}
- (G^{(k)}_{i_ki_1}G^{(k)}_{n i_{1}}+G^{(k)}_{i_{k} n}G^{(k)}_{i_1 i_1}) G^{(1)}_{n i_2}\dots G^{(k-1)}_{i_{k-1}i_{k}}=-m(z_{k})G^{(1)}_{n i_{2}}\dots G^{(k)}_{i_{k}n} + \stO{\frac{1}{N^{k/2+1/2}\sqrt{\eta\eta_k}}}
\end{align} whenever $n\not = i_{1}, i_{k}$. For these exceptional $n$ we shall again use the trivial $ N^{-k/2}\eta^{-1/2}$ estimate. For $a\not =k$ 
the summand in the second term of eq.~\eqref{eq:productRule} can always be estimated by \[- (G^{(a)}_{i_{a} i_{1}}G^{(a)}_{n i_{a+1}}+G^{(a)}_{i_{a} n}G^{(a)}_{i_{1} i_{a+1}}) G^{(1)}_{n i_{2}}\prod_{a\not= b=2}^kG^{(b)}_{i_{b}i_{b+1}}\prec \frac{1}{N^{k/2}\sqrt{\eta}} \] and this bound can be improved to \[- (G^{(a)}_{i_{a} i_{1}}G^{(a)}_{n i_{a+1}}+G^{(a)}_{i_{a} n}G^{(a)}_{i_{1} i_{a+1}}) G^{(1)}_{n i_{2}}\prod_{a\not= b=2}^kG^{(b)}_{i_{b}i_{b+1}}\prec \frac{1}{N^{k/2+1/2}\sqrt{\eta\eta_a}}, \] whenever $n\not\in\{i_{1},\dots,i_{k}\}$. 
Thus for most of the  $\landauO{N^{k+1}}$  terms in the sum in eq.~\eqref{eq:longTraceExp}  we have the improved bound, while
for  $\landauO{N^k}$ terms, where $n=i_{l}$ for some $l$, we use the weaker bound 
 and we find that 
 \begin{align}&\label{eq:scEGG}
N^{-k/2}\sum^\sim_{i_{1},\dots, i_{k}}\sum_n\frac{1}{N}\E\frac{\partial \left[G^{(1)}_{ni_2}G^{(2)}_{i_2i_3}\dots G^{(k)}_{i_ki_1}\right]}{\partial h_{i_1n}}\\&\qquad=\frac{1}{N^{k/2+1}z_{1} }\sum^\sim_{n,i_{1},\dots, i_{k}} \left[ -m(z_{1}) \E G^{(1)}_{i_{1} i_{2}}\dots G^{(k)}_{i_{k}i_{1}} -m(z_{k})\E G^{(1)}_{n i_{2}}\dots G^{(k)}_{i_{k}n} \right]+\frac{1}{z_{1}}\stO{\sum_{a=1}^k\frac{1}{\sqrt{N\eta \eta_{a}}}}.\nonumber
\end{align} It remains to estimate the error $R$. To do so we have to compute the second derivatives 
 \[\frac{\partial^2 \left[G^{(1)}_{ni_2}G^{(2)}_{i_2i_3}\dots G^{(k)}_{i_ki_1}\right]}{\partial h_{i_1n}^2}\] which is a polynomial in $G^{(l)}_{ab}$ for $l\in\{1,\dots,k\}$, $a,b\in\{i_1,\dots,i_k,n\}$ of total degree $k+2$ with  at most $2$ diagonal factors for $n\not\in\{i_1,\dots,i_k\}$, and otherwise with at most $3$ diagonal factors in every monomial. These factors each satisfy the entry-wise local  law \eqref{eq:localSCentry}, but now we need these
 estimates even uniformly for all $\abs{h_{i_1n}}\le N^{\tau-1/2}$ which does not directly follow from the
 concept of stochastic domination. To circumvent this technical issue, we need to explicitly display the
 dependence of the resolvents $G^{(l)}$ on $h_{i_1n}$.
 We therefore write $\widetilde H$ for the matrix $\widehat H$ 
 with the $(i_1, n)$ and $(n,i_1)$ entries set to $0$ and $\widetilde G^{(l)}=(\widetilde H-z_l)^{-1}$. Note that $\widetilde G^{(l)}$  is independent of $h_{i_1n}$.
Since $\widetilde G^{(l)}$ is the resolvent of a generalized Wigner matrix, from \cite{Erdos20121435, 2012arXiv1212.0164E}
 we have the usual resolvent estimates \eqref{eq:localSC}--\eqref{eq:localSCentry} for $\widetilde G^{(l)}$. 
Moreover, if $i_1\not=n$, then by the resolvent identity \begin{align*}G^{(l)}_{ab}=\widetilde G^{(l)}_{ab}-h_{i_1n}\left[\widetilde G^{(l)}_{an}\widetilde G^{(l)}_{i_1b}+\widetilde G^{(l)}_{ai_1}\widetilde G^{(l)}_{nb}\right]+h_{i_1n}^2\left[\widetilde G^{(l)}_{an}\widetilde G^{(l)}_{i_1n}G^{(l)}_{i_1b}+\widetilde G^{(l)}_{an}\widetilde G^{(l)}_{i_1i_1}G^{(l)}_{nb}+\widetilde G^{(l)}_{ai_1}\widetilde G^{(l)}_{nn}G^{(l)}_{i_1b}+\widetilde G^{(l)}_{ai_1}\widetilde G^{(l)}_{ni_1}G^{(l)}_{nb}\right]\end{align*}
and we can estimate \[\max_{a\not= b}\sup_{\abs{h_{i_1n}}\leq N^{-1/2+\tau}}G^{(l)}_{ab}\prec \frac{N^{\tau}}{\sqrt{N\eta_l}}, \qquad \max_{a}\sup_{\abs{h_{i_1n}}\leq N^{-1/2+\tau}}G^{(l)}_{aa}\prec 1\] whenever $\tau<1/12$ where we used the trivial bound $G_{ab}^{(l)}\leq 1/\eta_l\leq N^{2/3}$. On the other hand, if $i_1=n$, then we have \begin{align*}
G^{(l)}_{ab}=\widetilde G^{(l)}_{ab}-h_{nn}\widetilde G^{(l)}_{an}\widetilde G^{(l)}_{nb}+h_{nn}^2\widetilde G^{(l)}_{an}\widetilde G^{(l)}_{nn}G^{(l)}_{nb}
\end{align*} and therefore again \[\max_{a\not= b}\sup_{\abs{h_{nn}}\leq N^{-1/2+\tau}}G^{(l)}_{ab}\prec \frac{N^{\tau}}{\sqrt{N\eta_l}},
\qquad \max_{a}\sup_{\abs{h_{nn}}\leq N^{-1/2+\tau}}G^{(l)}_{aa}\prec 1\] whenever $\tau<1/12$. Therefore \[\sup_{\abs{h_{i_1n}}<N^{-1/2+\tau}}\abs{\frac{\partial^2 \left[G^{(1)}_{ni_2}G^{(2)}_{i_2i_3}\dots G^{(k)}_{i_ki_1}\right]}{\partial h_{i_1n}^2}}  \prec \sum_{a=1}^k\frac{N^{k\tau}N^{-k/2}}{\sqrt{\eta\eta_a}}\] and we can conclude 
\begin{align}\frac{1}{N^{k/2}z_1}\sum^\sim_{i_1,\dots, i_k}\sum_n\E \abs{h_{1_1n}}^3 \sup_{\abs{h_{i_1n}}<N^{-1/2+\tau}}\abs{\frac{\partial^2 \left[G^{(1)}_{ni_2}G^{(2)}_{i_2i_3}\dots G^{(k)}_{i_ki_1}\right]}{\partial h_{i_1n}^2}}\prec \sum_{a=1}^k\frac{N^{k\tau}N^{-1/2}}{\sqrt{\eta\eta_a}}.\label{eq:R}\end{align}
We can now pick $\tau=\min\{\frac{1}{12},\frac{\epsilon}{k}\}$ to have a final estimate of  order \[\sum_{a=1}^k \frac{N^{\epsilon}}{\sqrt{N\eta\eta_a}}\] for the error  originating from the last term in the truncated cumulant expansion \eqref{eq:cumulant}. The remaining error \begin{align}\E \abs{h_{i_1n}^3\1(\abs{h_{i_1n})}>N^{\tau-1/2}}\sup_{h_{i_1n}}\abs{\frac{\partial^2 \left[G^{(1)}_{ni_2}G^{(2)}_{i_2i_3}\dots G^{(k)}_{i_ki_1}\right]}{\partial h_{i_1n}^2}}\label{eq:R2}\end{align} is negligible for any fixed $k$ since the expectation is smaller than any power of $N^{-\tau}$
due to the arbitrary polynomial decay  \eqref{moments}.

  Putting together \eqref{eq:scEGG}, the identity \[z_1+m(z_1)+m(z_k)=\frac{m(z_1)m(z_k)-1}{m(z_1)}\] and the estimates on $R$ from \eqref{eq:R}--\eqref{eq:R2} we have shown that 
 \[N^{-k/2}\sum_{i_{1}\not=\dots\not= i_{k}}\E G^{(1)}_{i_{1}i_{2}}\dots G^{(k)}_{i_{k}i_{1}}=\frac{m(z_{1})}{1-m(z_{1})m(z_{k})}\landauO{\sum_{a=1}^k\frac{N^{\epsilon}}{\sqrt{N\eta \eta_{a}}}}=\stO{\sum_{a=1}^k\frac{N^{\epsilon}}{(\eta_{1}+\eta_{k})\sqrt{N\eta \eta_{a}}}}.\] 
  Since the lhs.~of this estimate is cyclic in $i_1,\dots,i_k$, we can replace $\eta_1+\eta_k$ in the error term by $\max_a\eta_a$.

For the proof of eq.~\eqref{eq:HigherTracesFree} we follow essentially the same steps but for the last $a=k-1$ term we find \begin{align*}
- (G^{(k-1)}_{i_{k-1}i_1}G^{(k-1)}_{n i_{k}}+G^{(k-1)}_{i_{k-1} n}G^{(k-1)}_{i_1 i_k}) G^{(1)}_{n i_2}\dots G^{(k-2)}_{i_{k-2}i_{k-1}}\prec\frac{1}{N^{k/2}\sqrt{\eta\eta_{k-1}}}
\end{align*}
instead of eq.~\eqref{eq:ak}. Consequently, eq.~\eqref{eq:scEGG} becomes \begin{align*}&
N^{-(k+1)/2}\sum_{i_{1}\not=\dots\not= i_{k}}\E G^{(1)}_{i_{1}i_{2}}\dots G^{(k-1)}_{i_{k-1}i_{k}}\\&\qquad=\frac{1}{N^{(k+1)/2+1}z_{1} }\sum_{n\not =i_{1}\not=\dots\not= i_{k}} \left[ -m(z_{1}) \E G^{(1)}_{i_{1} i_{2}}\dots G^{(k-1)}_{i_{k-1}i_{k}} \right]+\frac{1}{z_{1}}\stO{\sum_{a=1}^{k-1}\frac{N^{\epsilon}}{\sqrt{N\eta \eta_{a}}}}\nonumber
\end{align*}
from which eq.~\eqref{eq:HigherTracesFree} follows immediately.

For the last claim, note that none of the estimates above relied on the order of the indices of any $G^{(l)}$ and the same bound holds true in the case of any transpositions.

 The proof of the Hermitian case is similar, but the cumulant expansion has to be replaced by a complex variant (as in, e.g. \cite[Lemma 7.1]{2016arXiv160301499H}).
\end{proof}
Next, we note that the bounds \eqref{eq:HigherTraces}--\eqref{eq:HigherTracesFree} also hold true without taking expectations:
\begin{corollary}
In the setup of Lemma \ref{lemma:cycles}, for closed cycles of length $k\geq 2$ we have that \begin{align}
N^{-k/2}\sum^{\sim}_{i_1,\dots,i_k} G^{(1)}_{i_1i_2}\dots G^{(k-1)}_{i_{k-1}i_k}G^{(k)}_{i_ki_1}\prec \frac{1}{\left(\max_a\eta_a\right)\sqrt{N\eta_1\dots\eta_k }} \sum_{a=1}^k \frac{1}{\sqrt{\eta_a}} ,
\label{eq:HigherTracesR}
\end{align} 
and for open cycles of any length $k>1$ we have that 
\begin{align}
N^{-(k+1)/2}\sum^\sim_{i_1,\dots, i_{k}} G^{(1)}_{i_1i_2}\dots G^{(k-1)}_{i_{k-1}i_{k}}\prec \frac{1}{\sqrt{N\eta_1\dots\eta_{k-1} }} \sum_{a=1}^{k-1} \frac{1}{\sqrt{\eta_a}}.
\label{eq:HigherTracesFreeR}
\end{align} 
\end{corollary}
\begin{proof}
We note that the fluctuation averaging analysis from \cite[Proof of Prop.~5.3 in Sections 6--7]{2013AnHP...14.1837E} does not rely on the fact $z_1=\dots=z_k$ and therefore also applies to the present case.
\end{proof}
The following lemma shows an asymptotic Wick theorem for $X$'s, i.e. that higher moments of $X$ to leading order only involve pairings:
\begin{lemma}\label{lemma:pairings}
For $k\geq 2$ and $z_1,\dots,z_k\in\C$ with $\abs{\Im z_l}=\eta_l>0$ we have that
\begin{align} \E[X(z_1)\dots X(z_k)]&= \sum_{\pi\in P_2([k])}\prod_{\{a,b\}\in \pi}\E[X(z_a)X(z_b)]+\stO{\frac{1}{\sqrt{N\eta_1\dots\eta_k}}\sum_{a\not= b}\frac{1}{(\eta_a+\eta_b)\sqrt{\eta_a}}},
\end{align} where $[k]\defeq\{1,\dots,k\}$ and $P_2(L)$ are the partitions of a set $L$ into subsets of size $2$.
\end{lemma}
\begin{proof} For definiteness  we  prove the real symmetric case. Since the argument  relies on counting pairings, the proof of the complex Hermitian case is very similar and we omit it.
We have to compute
\begin{align*}
\E_1\prod_{l=1}^k\left[\sum_{i_l\not=j_l}h_{i_l}G_{i_{l}j_l}^{(l)}h_{j_l}+\sum_{i_l}\left(h_{i_l}^2-\frac{1}{N}\right)G^{(l)}_{i_li_l}\right]&=\sum_{L\subset [k]}\E_1\left[\left(\prod_{l\in L}\sum_{i_l\not=j_l}h_{i_l}G_{i_{l}j_l}^{(l)}h_{j_l}\right)\left(\prod_{l\not\in L}\sum_{i_l}\left(h_{i_l}^2-\frac{1}{N}\right)G^{(l)}_{i_li_l}\right)\right], 
\end{align*}
where $[k]=\{1,\dots,k\}$  and $\E_1=\E(\cdot|H^{(1)})=\E(\cdot|\widehat H$) and we recall that $G^{(l)}$ is independent of $h$. We already know from eq.~\eqref{eq:EXX} and Lemma \ref{lemma:trGG} that the leading order of this expression is at most $N^{-k/2}$. In order to have non-zero expectation we have to pair any $h_{i_l}$ and $h_{j_l}$ with at least some other $h_{i_m}$ or $h_{j_m}$. An easy counting argument using the bound $G^{(l)}_{i_lj_l}\prec (N\eta_l)^{-1/2}$ shows that for any $L\subset[k]$ the corresponding $L$-term is at most of order \[ N^{-(k+1)/2}\prod_{l\in L}\eta_l^{-1/2} \] whenever any three or more $h_i$'s are paired. This already shows that we can restrict our attention to pairings and in particular odd moments asymptotically are of lower order. 

 Starting from some $h_{i_l}$ with $l\not\in L$ we have to pair it either to another $h_{i_m}$ with $m\not\in L$, or some $h_{i_m}$ or $h_{j_m}$ with $m\in L$. In the former case we have a closed pairing with expectation \[\E_1\left[(h_{i_l}^2-1/N)(h_{i_l}^2-1/N)G_{i_li_l}^{(l)}G_{i_li_l}^{(m)}\right]=\frac{\sigma_4-1}{N^2}G_{i_li_l}^{(l)}G_{i_li_l}^{(m)}.\] In the latter case, say we paired $h_{i_l}$ to $h_{i_m}$, we have to continue the pairing process by pairing $h_{j_m}$ with another $h_{i_k}$ or $h_{j_k}$ with $k\in L$ etc., until we reach another $h_{i_n}$ with $n\not\in L$. This expression represents an open cycle as in \eqref{eq:HigherTracesFreeR} and is therefore subleading.

On the other hand, starting from some $h_{i_l}$ or $h_{j_l}$ with $l\in L$, and continue the pairings as in the previous paragraph until we pair to an $h_{i_m}$ with $m\not\in L$ which results in an open cycle as in \eqref{eq:HigherTracesFreeR} and is subleading. Therefore we only have to consider closed cycles of the pure $L$-type, from which, due to \eqref{eq:HigherTracesR}, only those of length $2$ are leading. That means that pairing $h_{i_l}$ to $h_{i_m}$ automatically forces a pairing of $h_{j_l}$ and $h_{j_m}$, and that a pairing of $h_{i_l}$ to $h_{j_m}$ automatically forces a pairing of $h_{j_l}$ and $h_{i_m}$. These give the leading contribution of \[ \E_1 \left[h_{i_l}G_{i_lj_l}^{(l)}h_{j_l}h_{j_l}G_{j_li_l}^{(m)}h_{i_l}+h_{i_l}G_{i_lj_l}^{(l)}h_{j_l}h_{i_l}G_{i_lj_l}^{(m)}h_{j_l}\right]=\frac{G_{i_lj_l}^{(l)}G_{i_lj_l}^{(m)}}{N^2}.\]
 The above findings allow us to conclude that
\begin{align*}
&N^{k/2}\E_1\left[\left(\prod_{l\in L}\sum_{i_l\not=j_l}h_{i_l}G_{i_{l}j_l}^{(l)}h_{j_l}\right)\left(\prod_{l\not\in L}\sum_{i_l}\left(h_{i_l}^2-\frac{1}{N}\right)G^{(l)}_{i_li_l}\right)\right]\\&\quad=N^{k/2}\E_1\left(\prod_{l\in L}\sum_{i_l\not=j_l}h_{i_l}G_{i_{l}j_l}^{(l)}h_{j_l}\right)\E_1\left(\prod_{l\not\in L}\sum_{i_l}\left(h_{i_l}^2-\frac{1}{N}\right)G^{(l)}_{i_li_l}\right) +\stO{\Psi}\\&\quad=N^{k/2}\left(\sum_{\pi\in P_2(L)}\prod_{\{a,b\}\in \pi}\sum_{i\not=j}\frac{G_{ij}^{(a)}G_{ij}^{(b)}+G_{ij}^{(a)}G_{ji}^{(b)}}{N^2}\right)\left(\sum_{\pi\in P_2([k]\setminus L)}\prod_{\{a,b\}\in \pi}\frac{\sigma_4-1}{N^2}\sum_{i} G^{(a)}_{ii} G^{(b)}_{ii}\right) +\stO{\Psi}\\&\quad=N^{k/2}\left(\sum_{\pi\in P_2(L)}\prod_{\{a,b\}\in \pi}\frac{2}{N}\frac{m(z_a)^2m(z_b)^2}{1-m(z_a)m(z_b)}\right)\left(\sum_{\pi\in P_2([k]\setminus L)}\prod_{\{a,b\}\in \pi}\frac{\sigma_4-1}{N}m(z_a)m(z_b)\right) +\stO{\Psi},
\end{align*}
where in the last step we used Lemma \ref{lemma:trGG} and we introduced the error term \[\Psi=\frac{1}{\sqrt{N\eta_1\dots\eta_k}}\sum_{a\not= b}\frac{1}{(\eta_a+\eta_b)\sqrt{\eta_a}}.\] We now recognize the last expression as the sum over products of pairs of $\E[X(z_a)X(z_b)]$, completing the proof.
\end{proof} 

We now have all ingredients to compute \begin{align*}
&\left(-\Im\int_{\R}\int_{\eta_0}^{10}g(z)\left[X(z)-\sqrt N h_{11}\right]\diff \eta\diff x\right)^k\\&\qquad=\sum_{j=0}^k \binom{k}{j} (\sqrt N h_{11})^{k-j} \left(\Im\int_{\R}\int_{\eta_0}^{10}g(z)\diff\eta\diff x\right)^{k-j}\left(-\Im\int_{\R}\int_{\eta_0}^{10}g(z)X(z)\diff \eta\diff x\right)^{j}. 
\end{align*}
Recall that $h_{11}$ and $X$ are independent. From Lemma \ref{lemma:pairings} we can conclude that \begin{align*}\E\left(-\Im\int_{\R}\int_{\eta_0}^{10}g(z)X(z)\diff \eta\diff x\right)^{j} &= \sum_{\pi\in P_2([j])}\left(2V_{f,1}+(\sigma_4-1)V_{f,2}\right)^{j/2} +\stO{N^{-1/6}}\\&= (j-1)!!\left(2V_{f,1}+(\sigma_4-1)V_{f,2}\right)^{j/2}+\stO{N^{-1/6}}\end{align*} for even $j$ and 
\[\E\left(-\Im\int_{\R}\int_{\eta_0}^{10}g(z)X(z)\diff \eta\diff x\right)^{j}=\stO{N^{-1/6}}\] for odd $j$. If $h_{11}$ follows a normal distribution, then $\E h_{11}^{k-j}=(k-j-1)!!\left(s_{11}/N\right)^{(k-j)/2}$ whenever $k-j$ is even and $\E h_{11}^{k-j}=0$, otherwise. Therefore, since \[(j-1)!!(k-j-1)!!\binom{k}{j}=(k-1)!!\binom{k/2}{j/2}\] for even $j,k$, we have that \begin{align}\E\left(-\Im\int_{\R}\int_{\eta_0}^{10}g(z)\left[X(z)-\sqrt N h_{11}\right]\diff \eta\diff x\right)^k =(k-1)!! \left[2V_{f,1}+(\sigma_4-1)V_{f,2}+s_{11}V_{f,3} \right]^{k/2}+\stO{N^{-1/6}}\label{momentsReal}\end{align} whenever $k$ is even and \[\E\left(-\Im\int_{\R}\int_{\eta_0}^{10}g(z)\left[X(z)-\sqrt N h_{11}\right]\diff \eta\diff x\right)^k =\stO{N^{-1/6}} \] otherwise.

For the case of complex Hermitian $H$ we can follow the same argument and ultimately find that eq.~\eqref{momentsReal} becomes 
\[\E\left(-\Im\int_{\R}\int_{\eta_0}^{10}g(z)\left[X(z)-\sqrt N h_{11}\right]\diff \eta\diff x\right)^k = (k-1)!! \left[V_{f,1}+\abs{\sigma_2}^2V_{\sigma_2}+(\sigma_4-1)V_{f,2}+s_{11}V_{f,3} \right]^{k/2}+\stO{N^{-1/6}}.\]
Finally, we remark that the same proof also works in the case of $\widetilde f_N$ and we basically only have to replace $V_{f,3}$ by $\widetilde V_{f,3}$. 

\appendix
\section{Comparison to Gaussian Free Field}\label{sec:GFFcomp}
In this section we investigate to what extent our main result on the Gaussian fluctuation of
linear statistics of $H$ and its  minor $\widehat H$ is consistent with the Gaussian free field (GFF) limit
proved in \cite{Bor1, Bor2, 2015arXiv150405933L} for real symmetric matrices. In these papers the joint fluctuations of the spectral counting functions of 
minors were shown to converge to a  GFF in the large $N$ limit, assuming that the sizes of the minors asymptotically differed by $cN$.
Our result corresponds to the difference of the linear statistics of two minors whose sizes differ only by one. The fluctuation 
is only of order $N^{-1/2}$ and it is not visible on the macroscopic scale studied in \cite{Bor1, Bor2, 2015arXiv150405933L}.
Nevertheless, one may \emph{formally} apply these macroscopic  result  to our case.
Here we show that this naive extension indeed provides the correct order of magnitude  and also the correct variance of the fluctuations, but does not identify their precise distribution.

 For comparability with \cite{Bor1,Bor2,2015arXiv150405933L} assume a constant variance on the diagonal and constant fourth moment on the off-diagonal, i.e., $\E h_{ii}^2=\E h_{11}^2=s_{11}/N$ and $\E h_{ij}^4=\sigma_4/N^2$ for all $i\not = j$. 
 First we recall the main result of  \cite{2015arXiv150405933L} which   is based on \cite{Bor1}, where the corresponding formula was first proved for monomial
 test functions. Given an $N\times N$ Wigner matrix $H$, we denote the consecutive 
 lower right minors by $H_n\defeq (H_{jk})_{j,k=N-n+1}^N$. A special case of Theorem 2.2 of \cite{2015arXiv150405933L} then asserts that for any $f\in H^{5.5+\epsilon} (\R)$  and for any $x,y\in (0, 1]$, the covariance of linear statistics of 
two nested minors of size $Nx$ and $Ny$  is asymptotically  given by 
\begin{align} \label{GFF}
 C_f(x,y)\defeq & \lim_{N\to\infty}\Cov{\Tr f(H_{[xN]}),\Tr f(H_{[yN]})} \\=& \frac{1}{\pi^2} \oint_{\substack{\abs{z}^2 = x\\ \Im z>0}}\oint_{\substack{\abs{w}^2 = y\\ \Im w>0}} f'\left( z + \frac{x}{z} \right) f'\left( w + \frac{y}{w} \right)\log\abs{\frac{x\wedge y -zw}{x\wedge y -z\overline{w}}} \left( 1 - \frac{x}{z^2} \right)\left( 1 - \frac{y}{w^2} \right)\diff w\diff z \nonumber \\ &\qquad + \frac{s_{11}-2}{x\vee y} \left(\frac{1}{2\pi}\int_{-2\sqrt{x}}^{2\sqrt{x}} \frac{sf(s)}{\sqrt{4x-s^2}}\diff s\right)\left(\frac{1}{2\pi}\int_{-2\sqrt{y}}^{2\sqrt{y}} \frac{tf(t)}{\sqrt{4y- t^2}}\diff t\right) \nonumber \\
 & \qquad + \frac{\sigma_4-3}{2(x\vee y)^2} \left(\int_{-2\sqrt{x}}^{2\sqrt{x}} \frac{2 x-s^2}{\pi\sqrt{4x-s^2}}f(s)\diff s\right)\left(\int_{-2\sqrt{y}}^{2\sqrt y} \frac{2 y-t^2}{\pi\sqrt{4y-t^2}}f(t)\diff t\right)\nonumber
 \end{align} 
  where the  $z$ and $w$  integrations in the first term are  on the semicircular arcs in counterclockwise order. 
  
  Recalling our previous notation $H=H_N$ and $\widehat H = H_{N-1}$, 
  in our Theorem~\ref{thm:mainThm} we derived a formula for the rescaled variance  \begin{align*}D_{N,f}\defeq &N \Var [\Tr f(H_N)-\Tr f(H_{N-1})]\\=&N\big[ \Cov{\Tr f(H_N),\Tr f(H_{N})}-\Cov{\Tr f(H_{N}),\Tr f(H_{N-1})}\\ &\qquad\qquad\qquad\qquad-\Cov{\Tr f(H_{N-1}),\Tr f(H_{N})}+\Cov{\Tr f(H_{N-1}),\Tr f(H_{N-1})} \big], \end{align*}
 which corresponds to 
 \[
 N\Big[C_f(1,1)-C_f\big(1,1-\frac{1}{N}\big)-C_f\big(1-\frac{1}{N},1-\frac{1}{N}\big)+C_f\big(1-\frac{1}{N},1-\frac{1}{N}\big)\Big],
 \]
  suggesting that we should compare our result to the limit 
  \begin{equation}\label{Ddef}
  D_f\defeq\lim_{\epsilon\to 0}\frac{C_f(1,1)-C_f(1,1-\epsilon)-C_f(1-\epsilon,1)-C_f(1-\epsilon,1-\epsilon)}{\epsilon}.
  \end{equation}
  Note that this latter formula is the renormalized derivative of the Gaussian free field $\phi_x(f)$ with covariance $C_f(x,y)$ at $x=1$:
  \[ D_f =   \lim_{\epsilon\to 0} \Var \frac{\phi_1(f) - \phi_{1-\epsilon}(f)}{\sqrt{\epsilon}}.
  \]
  In the following theorem we compare the field 
  \[\psi_x^{(N)}(f)\defeq \Tr f(H_{[xN]})-\E \Tr f(H_{[xN]})\]
 defined by our linear eigenvalue statistics
to   
  the Gaussian free field $\phi_x(f)$.
  \begin{theorem}\label{thm:DD} Let $H$ be real symmetric Wigner matrices satisfying the conditions from Theorem \ref{thm:mainThm} and additionally assume that $\E h_{ii}^2=\E h_{11}^2=s_{11}/N$ and $\E h_{ij}^4=\sigma_4/N^2$ for all $i\not=j$. Then for any $f\in H^2(\R)$
  the centered random variables
   \[ X_f \defeq \lim_{\epsilon\to0} \frac{\phi_1(f)-\phi_{1-\epsilon}(f)}{\sqrt\epsilon}
   \qquad \mbox{and}\qquad Y_f \defeq \lim_{N\to\infty} \frac{\psi_{1}^{(N)}(f)-\psi_{1-1/N}^{(N)}(f)}{\sqrt{1/N}}\] are well defined (the limit is in
  distribution sense) and they
   have the same  variance 
   \begin{equation}\label{EXY}
   \E X_f^2 = \E Y_f^2 =2\int_{-2}^2 f'(s)^2\rho(s)\diff s + (\sigma_4-3)\left(\int_{-2}^2sf'(s)\rho(s)\diff s\right)^2+(s_{11}-2)\left( \int_{-2}^2f'(s)\rho(s)\diff s\right)^2.
   \end{equation}  
   Moreover, the distributions of $X_f$ and $Y_f$ agree if and only if $h_{11}$ follows a Gaussian distribution. \end{theorem}
\begin{proof} 
The variance formula for $Y_f$ follows immediately from Theorem \ref{thm:mainThm}. 

In order to prove that $X_f$ is well defined and follows a Gaussian distribution, it suffices to check that $D_f$ is finite. To do so, we treat the three terms of $C_f(x,y)$ from \eqref{GFF} separately, which for convenience we call $C_f(x,y)=C_f^{(1)}(x,y)+C_f^{(2)}(x,y)+C_f^{(3)}(x,y)$. It is easy to check that \[\lim_{\epsilon\to 0}\frac{C_f^{(2)}(1,1)-C_f^{(2)}(1,1-\epsilon)-C_f^{(2)}(1-\epsilon,1)-C_f^{(2)}(1-\epsilon,1-\epsilon)}{\epsilon}= (s_{11}-2)\left(\int_{-2}^2f'(s)\rho(s)\diff s \right)^2\] and that
\[\lim_{\epsilon\to 0}\frac{C_f^{(3)}(1,1)-C_f^{(3)}(1,1-\epsilon)-C_f^{(3)}(1-\epsilon,1)-C_f^{(3)}(1-\epsilon,1-\epsilon)}{\epsilon}= (\sigma_4-3)\left(\int_{-2}^2 s f'(s)\rho(s)\diff s \right)^2.\]
For the computation of $C_f^{(1)}(x,y)$ we now substitute $z=\sqrt x e^{i\phi}$ and $w=\sqrt y e^{i\psi}$ with $\phi,\psi\in[0,\pi]$, so that
 \begin{align*}
C_f^{(1)}(x,y)=\frac{4\sqrt{xy}}{\pi^2}\int_{0}^\pi\int_0^\pi f'(2\sqrt{x}\cos \phi) f'(2\sqrt{y}\cos \psi) \log\abs{\frac{x\wedge y -\sqrt{xy}e^{i(\phi+\psi)} }{x\wedge y -\sqrt{xy}e^{i(\phi-\psi)}}} \sin\phi \sin\psi \diff \psi\diff\phi
 \end{align*}
 and after a further substitution of $2\sqrt x \cos\phi = s$ and $2\sqrt y\cos \psi = t$ and simple algebraic manipulation we arrive at \begin{align*}C_f^{(1)}(x,y)&=\frac{1}{\pi^2} \int_{-2\sqrt x}^{2\sqrt x}\int_{-2\sqrt y}^{2\sqrt y} f'(s)f'(t)\arctanh{\frac{\sqrt{(4x-s^2)(4y-t^2)}}{2(x+y)- st}}\diff t\diff s.\end{align*}
To keep the notation relatively short we now introduce \[a_{x,y}(s,t)\defeq \arctanh{\frac{\sqrt{(4x-s^2)(4y-t^2)}}{2(x+y)- st}}=\arctanh \sqrt{1-\frac{(x-y)^2+(t-s)(xt-ys)}{(x+y)^2-(x+y)st+s^2t^2/4}}\] and we claim that \[\frac{a_{1,1}(s,t)-a_{1,1-\epsilon}(s,t)-a_{1-\epsilon,1}(s,t)+a_{1-\epsilon,1-\epsilon}(s,t)}{\epsilon}\approx \pi\delta(s-t)\sqrt{4-t^2}  \] for any fixed $s,t\in[-2,2]$ in the $\epsilon\to0$ limit. Firstly, one readily checks that when $\abs{s-t}\gg\epsilon$, then 
\begin{align*}
\lim_{\epsilon\to0}\frac{a_{1,1}(s,t)-a_{1,1-\epsilon}(s,t)-a_{1-\epsilon,1}(s,t)+a_{1-\epsilon,1-\epsilon}(s,t)}{\epsilon}=0.
\end{align*} Secondly, when $\abs{x-y}\leq \epsilon$ and $\abs{s-t}\leq M\epsilon$ for some large but fixed $M$, then a series expansion gives 
\[a_{x,y}(s,t)=\log 2-\frac{1}{2}\log \frac{(x-y)^2+(t-s)(xt-ys)}{(x+y)^2-(x+y)st+s^2t^2/4} - \frac{1}{4}\frac{(x-y)^2+(t-s)(xt-ys)}{(x+y)^2-(x+y)st+s^2t^2/4}+\landauO{\epsilon^2}, \]
assuming, additionally,  that $\abs{s}\le 2\sqrt{x}(1-\delta)$,  $\abs{t}\le 2\sqrt{y}(1-\delta)$
with some fixed $\delta>0$.  It can now be checked via an explicit integration that \[\int_{\abs{s-t}<M\epsilon} \frac{a_{1,1}(s,t)-a_{1,1-\epsilon}(s,t)-a_{1-\epsilon,1}(s,t)+a_{1-\epsilon,1-\epsilon}(s,t)}{\epsilon} \diff s = \pi\sqrt{4-t^2} +\landauO{\epsilon}\]
for fixed $t$, proving the claim. We can conclude that \begin{align*}\frac{C_f^{(1)}(1,1)-C_f^{(1)}(1,1-\epsilon)-C_f^{(1)}(1-\epsilon,1)-C_f^{(1)}(1-\epsilon,1-\epsilon)}{\epsilon}&=\frac{1}{\pi^2}\int_{-2}^2\int_{-2}^2 f'(s)f'(t)\delta(s-t) \pi \sqrt{4-t^2}\diff s\diff t+\landauO{\epsilon}\\ &= 2\int_{-2}^2 f'(t)^2\rho(t)\diff t+\landauO{\epsilon},\end{align*} where we used that $f'\in L^2$ and therefore the integral over the neglected area where $\abs{s}>2\sqrt{x}(1-\delta)$ or $\abs{t}>2\sqrt{y}(1-\delta)$ does not contribute to leading order. Thus \[D_f=2\int_{-2}^2 f'(s)^2\rho(s)\diff s + (\sigma_4-3)\left(\int_{-2}^2sf'(s)\rho(s)\diff s\right)^2+(s_{11}-2)\left( \int_{-2}^2f'(s)\rho(s)\diff s\right)^2,\] completing the proof of \eqref{EXY}. In particular, the limit defining $X_f$ exists and is Gaussian. 
Finally, the existence of the limit defining $Y_f$ follows from the moment calculations in section \ref{sec:moments} and assumption \eqref{moments} on the moments of $h_{11}$ that together also guarantee tightness. This completes the proof of
 the theorem.
\end{proof}

\bibliographystyle{abbrv} 
\bibliography{ref} 
 
\end{document}